\documentclass[11pt]{article}
\usepackage{amsmath}
\usepackage{amsfonts}
\usepackage{xcolor}
\usepackage{graphicx}

\newtheorem{theorem}{Theorem}

\newtheorem{corollary}[theorem]{Corollary}

\newtheorem{lemma}[theorem]{Lemma}

\newtheorem{proposition}[theorem]{Proposition}
\newtheorem{remark}[theorem]{Remark}

\newenvironment{proof}[1][Proof]{\textbf{#1.} }{\ \rule{0.5em}{0.5em}}
\begin{document}

\author{George Avalos \\
Department of Mathematics, University of Nebraska-Lincoln, USA \and Pelin G. Geredeli \\
Department of Mathematics, Iowa State University, USA \and Boris Muha \\
Department of Mathematics, Faculty of Science, University of Zagreb, Croatia}
\title{ Wellposedness, Spectral Analysis and Asymptotic Stability of a Multilayered Heat-Wave-Wave System
}
\maketitle

\begin{abstract}
In this work we consider a multilayered heat-wave system where a 3-D heat equation is coupled with a 3-D wave equation via a 2-D interface whose dynamics is described by a 2-D wave equation. This system can be viewed as a simplification of a certain fluid-structure interaction (FSI) PDE model where the structure is of composite-type; namely it consists of a \textquotedblleft thin\textquotedblright\ layer and a \textquotedblleft
thick\textquotedblright\ layer. We associate the wellposedness of the system with a strongly
continuous semigroup and establish its asymptotic decay.

Our first result is semigroup well-posedness for the (FSI) PDE dynamics. Utilizing here a Lumer-Phillips approach, we show that the fluid-structure system generates a $C_0$-semigroup on a chosen finite energy space of data. As our second result, we prove that the solution to the (FSI) dynamics generated by the $C_0$-semigroup tends asymptotically to the zero state for all initial data. That is, the semigroup of the (FSI) system is strongly stable. For this stability work, we analyze the spectrum of the generator $\mathbf A$ and show that the spectrum of $\mathbf A$  does not intersect the imaginary axis.  \vskip.3cm
\noindent \textbf{Key terms:} Fluid-structure interaction, heat-wave system, well-posedness, semigroup, strong stability
\end{abstract}

\bigskip

\section{Introduction}
\subsection{Motivation and Literature}
This work is motivated by a longstanding interest in the analysis of fluid-structure interaction (FSI) partial differential equation (PDE) dynamics. Such FSI problems deal with multi-physics systems consisting of fluid and structure PDE components. These systems are ubiquitous in nature and have many applications, e.g., in biomedicine \cite{FSIforBIO} and aeroelasticity \cite{Dowell15}. However, the resulting PDE systems are very complicated (due to nonlinearities, moving boundary phenomena and hyperbolic-parabolic coupling) and despite extensive research activity in last 20 years, the comprehensive analytic theory for such systems is still not available. Accordingly, by way of obtaining a better understanding of FSI dynamics, it would seem natural to consider those FSI PDE models, which although constitute a simplification of sorts, yet retain their crucial novelties and intrinsic difficulties. For example, in the past, coupled heat-wave PDE systems (and variations thereof) have been considered for study: the heat equation component is regarded as a simplification of the fluid flow component of the FSI dynamics; the wave equation component is regarded as a simplification of the structural (elastic) component; see e.g., [\cite{lions1969quelques}, Section 9] and \cite{RauchZhangZuazua}. See also the works \cite{Du, A-T, Barbu, Chambolle, Courtand}, in which the fluid PDE component of fluid-structure interactions is governed by Stokes or Navier-Stokes flow.

Here we consider a multilayered version of such heat-wave system; where the coupling of the 3-D heat and the 3-D wave equations is realized via an additional 2-D wave equation on the boundary interface. This is a simplified (yet physically relevant) version of a benchmark fluid-component structure PDE model which was introduced in \cite{SunBorMulti}. This particular FSI problem was principally motivated by the mathematical modeling of vascular blood flow: such modeling PDE dynamics will account for the fact that the blood-transporting vessels are generally composed of several layers, each with different mechanical properties and are moreover separated by the thin elastic laminae (see \cite{multi-layered} for more details). In order to mathematically model these biological features, the multilayered structural component of such FSI dynamics is governed by a 3-D wave-2-D wave PDE system. For the physical interpretation and derivation of such coupled "thick-thin" structure models we refer reader to \cite[Chapter 2]{CiarletBook2} and references within.

As we said, although the present multilayered heat- wave- wave system constitutes a simplification somewhat of the FSI model in \cite{SunBorMulti} -- in particular, the 2-D wave equation takes the place of a fourth order plate or shell PDE -- our results remain valid if we replace the 2-D wave equation with the corresponding linear fourth order equation. Within the context of the present multilayered heat-wave-wave coupled system, we are interested in asymptotic behavior of the solutions, and regularization effects of the fluid dissipation and coupling via the elastic interface, inasmuch as there is a dissipation of the natural energy of the heat-wave-wave PDE system – with this dissipation coming strictly from the heat component of the FSI dynamics – it is a reasonable objective to determine if this thermal dissipation actually gives rise to asymptotic decay (at least) to all three PDE solution components: That is, we seek to ascertain longtime decay of both 3-D and 2-D wave solution components, as well as the heat solution component. Such a strong stability can be seen as a measure of the "strength" of the coupling condition. For the classical heat-wave system (without the 2-D wave equation on the interface) this question is by now rather well understood and precise decay rates are well known (see \cite{AvalosLasieckaTriggiani16,Batty19} and references within.) (We should emphasize that the high-frequency oscillations in the structure are not efficiently dissipated and therefore there is no exponential decay of the energy.)

Our present investigation into the multilayered wave-heat systems is motivated in part by \cite{SunBorMulti} which considered a nonlinear FSI comprised by 2-D (thick layer) wave equation and 1-D wave equation (thin layer) coupled to a 2-D fluid PDE across a boundary interface. For these dynamics, wellposedness was established in \cite{SunBorMulti}, in part by exploiting an underlying regularity which was available by the presence of said wave equation. (Such regularizing effects were observed numerically in \cite{multi-layered} and precisely quantified in the sense of Sobolev for a 1-D FSI system in \cite{BorisSimplifiedFSI}. For similar regularizing effects in the context of hyperbolic-hyperbolic PDE couplings, we refer to \cite{HansenZuazua, KochZauzua, LescarretZuazua15}.) By way of gaining a better qualitative understanding of FSI systems, such as those in \cite{SunBorMulti}, we here embark upon an investigation of the aforesaid 3-D heat-2-D wave-3-D wave coupled PDE system; in particular, we will establish the semigroup wellposedness and asymptotic decay to zero of the underlying energy of this FSI. These objectives of wellposedness and decay will entail a precise understanding of the role played by the coupling mechanisms on the elastic interface and by the fluid dissipation. In future work, we will investigate possible regularizing effects, at least for certain polygonal configurations of the boundary interface.

We finish this section by giving a brief literature review, in addition to the ones mentioned above. FSI models have been very active and broad area of research in the last two decades and therefore here we avoid presenting a full literature review: we merely mention here a few recent monographs and review works \cite{FSIforBIO,bodnar2017particles,SunnyStentsSIAM,GazzolaReview,MTEFSI,RichterBook}, where interested reader can find further references. The study of various simplified FSI models which manifest parabolic-hyperbolic coupling has a long history going back at least to [\cite{lions1969quelques}, Section 9], where the Navier-Stokes equations are coupled with the wave equation along a fixed interface. However, even in the linear case the presence of the pressure term gives rise to significant mathematical challenges in developing the semigroup wellposednes theory \cite{AvalosTriggiani09}. Thus, the heat-wave system has been extensively studied in last decade as a suitable simplified model for stability analysis of parabolic-hyperbolic coupling occurring in FSI systems, see e.g. \cite{A-B, AvalosTriggiani2,Duyckaerts,Fathallah,ZhangZuazuaARMA07} and references within. To the best of our knowledge there are still no results about strong stability of FSI systems with an elastic interface.

\subsection{PDE Model}

Let the fluid geometry $\Omega _{f}$ $\subseteq \mathbb{R}^{3}$
be a Lipschitz, bounded domain. The structure domain $\Omega _{s}$ $%
\subseteq \mathbb{R}^{3}$ will be \textquotedblleft completely
immersed\textquotedblright\ in $\Omega _{f}$; with $\Omega _{s}$ being a
convex polyhedral domain.
\begin{center}
 \includegraphics[scale=0.4]{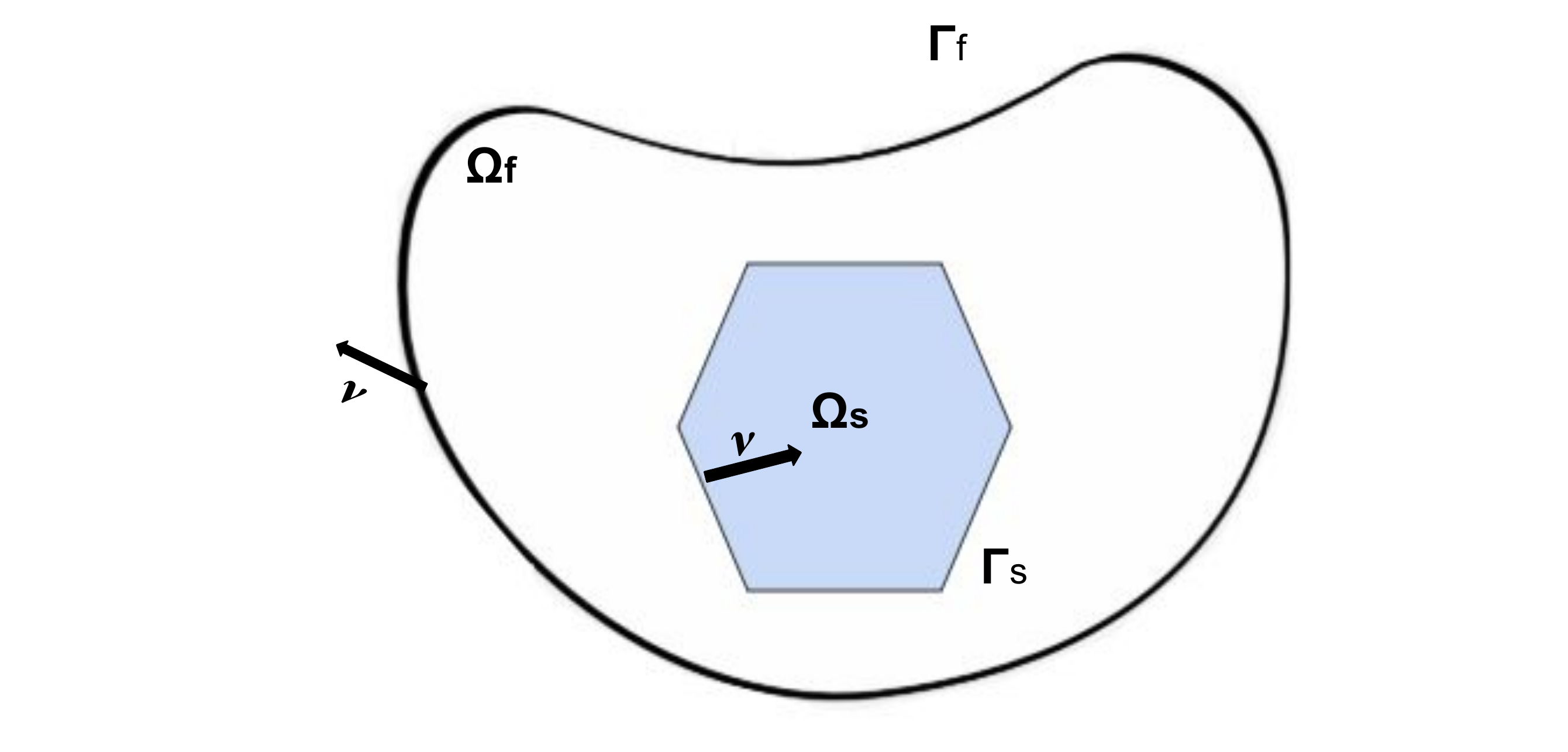}
 \end{center} 
 \begin{center}
 \textbf{Figure: Geometry of the FSI Domain}
 \end{center}
 In the figure, $\Gamma _{f}$ is the part of boundary of $\partial \Omega
_{f} $ which does not come into contact with $\Omega _{s}$; $\Gamma
_{s}=\partial \Omega _{s}$ is the boundary interface between $\Omega _{f}$
and $\Omega _{s} $ wherein the coupling between the two distinct fluid and
elastic dynamics occurs. (And so, $\partial \Omega _{f}=\Gamma _{s}\cup
\Gamma _{f}$.) We have that 
\begin{equation}
\Gamma _{s}=\cup _{j=1}^{K}\overline{\Gamma }_{j},  \label{1}
\end{equation}%
where $\Gamma _{i}\cap \Gamma _{j}=\emptyset $, for $i\neq j$. It is further
assumed that each $\Gamma _{j}$ is an open polygonal domain.

\medskip

Moreover, $n_{j}$ will denote the unit normal vector which is exterior to $%
\partial \Gamma _{j}$, $1\leq j\leq K$. With respect to this geometry, the $%
\mathbb{R}^{3}$ wave--$\mathbb{R}^{2}$ wave--$\mathbb{R}^{3}$ heat
interaction PDE model is given as follows:\\%
For $i\leq j\leq K,$
\begin{equation}
\left\{ 
\begin{array}{l}
u_{t}-\Delta u=0\text{ \ \ \ in \ }(0,T)\times \Omega _{f} \\ 
u|_{\Gamma _{f}}=0\text{ \ \ \ on \ }(0,T)\times \Gamma _{f};%
\end{array}%
\right.  \label{2a}
\end{equation}

\begin{equation}
\left\{ 
\begin{array}{l}
\frac{\partial ^{2}}{\partial t^{2}}h_{j}-\Delta h_{j}+h_{j}=\frac{\partial w%
}{\partial \nu }|_{\Gamma _{j}}-\frac{\partial u}{\partial \nu }|_{\Gamma
_{j}}\text{ \ \ \ on \ }(0,T)\times \Gamma _{j} \\ 
h_{j}|_{\partial \Gamma _{j}\cap \partial \Gamma _{l}}=h_{l}|_{\partial
\Gamma _{j}\cap \partial \Gamma _{l}}\text{ on \ }(0,T)\times (\partial
\Gamma _{j}\cap \partial \Gamma _{l})\text{, for all }1\leq l\leq K\text{
such that }\partial \Gamma _{j}\cap \partial \Gamma _{l}\neq \emptyset \\ 
\left. \dfrac{\partial h_{j}}{\partial n_{j}}\right\vert _{\partial \Gamma
_{j}\cap \partial \Gamma _{l}}=-\left. \dfrac{\partial h_{_{l}}}{\partial
n_{l}}\right\vert _{\partial \Gamma _{j}\cap \partial \Gamma _{l}}\text{on \ 
}(0,T)\times (\partial \Gamma _{j}\cap \partial \Gamma _{l})\text{, for all }%
1\leq l\leq K\text{ such that }\partial \Gamma _{j}\cap \partial \Gamma
_{l}\neq \emptyset .\text{\ }%
\end{array}%
\right.  \label{2.5b}
\end{equation}%
\begin{equation}
\left\{ 
\begin{array}{l}
w_{tt}-\Delta w=0\text{ \ \ \ on \ }(0,T)\times \Omega _{s} \\ 
w_{t}|_{\Gamma _{j}}=\frac{\partial }{\partial t}h_{j}=u|_{\Gamma _{j}}\text{
\ \ \ on \ }(0,T)\times \Gamma _{j}\text{, \ for }j=1,...,K%
\end{array}%
\right.  \label{2d}
\end{equation}%
\begin{equation}
\lbrack u(0),h_{1}(0),\frac{\partial }{\partial t}h_{1}(0),...,h_{K}(0),%
\frac{\partial }{\partial t}%
h_{K}(0),w(0),w_{t}(0)]=[u_{0},h_{01},h_{11},...,h_{0K},h_{1K},w_{0},w_{1}].
\label{IC}
\end{equation}%
Equation \eqref{2.5b}$_1$ is the dynamic coupling condition and represents a balance of forces on $\Gamma_j$. The left-hand side comes from the inertia and elastic energy of the thin structure, while the right-hand side accounts for the contact forces coming from the 3-D structure and the fluid, respectively. The last term of the left-hand side is added to ensure the uniqueness of the solution and physically means that the structure is anchored and therefore the displacement does not have a translational component. The coupling conditions \eqref{2.5b}$_2$ and \eqref{2.5b}$_3$ represent continuity of the displacement and contact force along the interface between sides $\Gamma_i$ and $\Gamma_l$, respectively. Equation \eqref{2d}$_2$ is a kinematic coupling condition and accounts for continuity of the velocity across the interface $\Gamma_j$. It corresponds to the no-slip boundary condition in fluid mechanics.
Note that the boundary condition in (\ref{2d}) implies that for $t>0$,%
\begin{equation*}
w(t)|_{\Gamma _{j}}-h_{j}(t)=w(0)|_{\Gamma _{j}}-h_{j}(0),\text{ \ \ for }%
j=1,...,K.
\end{equation*}%
Accordingly, the associated space of initial data $\mathbf{H}$ incorporates
a compatibility condition. Namely, 
\begin{equation}
\begin{array}{l}
\mathbf{H}=\{[u_{0},h_{01},h_{11},...,h_{0k},h_{1k},w_{0},w_{1}]\in
L^{2}(\Omega _{f})\times H^{1}(\Gamma _{1})\times L^{2}(\Gamma _{1})\times
... \\ 
\text{ \ \ \ \ \ \ \ \ \ \ \ }\times H^{1}(\Gamma _{K})\times L^{2}(\Gamma
_{K})\times H^{1}(\Omega _{s})\times L^{2}(\Omega _{s})\text{, \ such that
for each }1\leq j\leq K\text{:\ (i) }w_{0}|_{\Gamma _{j}}=h_{0j}\text{; } \\ 
\text{ \ \ \ \ \ \ \ \ }\left. \text{(ii) }h_{0j}|_{\partial \Gamma _{j}\cap
\partial \Gamma _{l}}=h_{0l}|_{\partial \Gamma _{j}\cap \partial \Gamma _{l}}%
\text{ on \ }\partial \Gamma _{j}\cap \partial \Gamma _{l}\text{, for all }%
1\leq l\leq K\text{ such that }\partial \Gamma _{j}\cap \partial \Gamma
_{l}\neq \emptyset \right\} .%
\end{array}
\label{H}
\end{equation}%
Because of the given boundary interface compatibility condition, $\mathbf{H}$
is a Hilbert space with the inner product%
\begin{eqnarray}
(\Phi _{0},\widetilde{\Phi }_{0})_{\mathbf{H}} &=&(u_{0},\widetilde{u}%
_{0})_{\Omega _{f}}+\sum\limits_{j=1}^{K}(\nabla h_{0j},\nabla \widetilde{h}%
_{0j})_{\Gamma _{j}}+\sum\limits_{j=1}^{K}(h_{0j},\widetilde{h}%
_{0j})_{\Gamma _{j}}  \notag \\
&&+\sum\limits_{j=1}^{K}(h_{1j},\widetilde{h}_{1j})_{\Gamma _{j}}+(\nabla
w_{0},\nabla \widetilde{w}_{0})_{\Omega _{s}}+(w_{1},\widetilde{w}%
_{1})_{_{\Omega _{s}}},  \label{Hilbert}
\end{eqnarray}%
where%
\begin{equation}
\Phi _{0}=\left[ u_{0},h_{01},h_{11},...,h_{0K},h_{1K},w_{0},w_{1}\right]
\in \mathbf{H}\text{; \ }\widetilde{\Phi }_{0}=\left[ \widetilde{u}_{0},%
\widetilde{h}_{01},\widetilde{h}_{11},...,\widetilde{h}_{0K},\widetilde{h}%
_{1K},\widetilde{w}_{0},\widetilde{w}_{1}\right] \in \mathbf{H}.
\label{stat}
\end{equation}

\subsection{Novelty and Challenges}
The novelty of this work is that we consider an FSI model in which the interface is elastic and has mass. This is the simplest model 3-D of the interaction of the fluid with the composite structure which retains basic mathematical properties of the physical model. To the best of our knowledge this is the first result about asymptotic behavior of solution to such problems. We work in setting were the structure domain is polyhedron and dynamics of each polygon side of the boundary is governed by the 2-D linear wave equation. The wave equations are coupled via dynamic and kinematic coupling conditions over the common boundaries. We choose this setting because it will directly translate to numerical analysis of the problem. This work is an important first step to a finer analysis of the asymptotic decay (e.g. decay rates) and regularity properties of the solutions, and to better understanding of the influence of the elastic interface with mass to the qualitative properties of the solutions. 

By way of establishing the semigroup wellposedness of the multilayered FSI
model \eqref{2a}-\eqref{IC} -- i.e., Theorem 1 below -- we will show that the associated
generator $\mathbf{A}$, defined by \eqref{4a} and (A.i)-(A.iv) below, is maximal
dissipative, and so generates a $C_{0}$-semigroup of contractions on the
natural Hilbert space of finite energy \eqref{6}. The presence of the
\textquotedblleft thin layer\textquotedblright\ wave equation on $\Gamma _{j}
$, $1\leq j\leq K$, complicates this wellposedness work, vis-\`{a}-vis the
situation which prevails for the previous 3-D heat-3-D wave models in
\cite{AvalosLasieckaTriggiani16, AvalosTriggiani2, Duyckaerts, RauchZhangZuazua, ZhangZuazuaARMA07} for which a relatively straight invocation of
the Lax-Milgram Theorem suffices to establish the maximality of the
associated FSI generator. In the present work, we will likewise apply
Lax-Milgram in order to ultimately show the condition $Range(\lambda I-%
\mathbf{A})=\mathbf{H}$ -- where $\lambda >0$ positive; in particular,
Lax-Milgram will be applied for the solvability of a certain variational
equation, relative to elements in a certain subspace of $H^{1}(\Omega
_{f})\times H^{1}(\Gamma _{1})\times ...\times H^{1}(\Gamma _{K})\times
H^{1}(\Omega _{s})$. (See \eqref{8} below). This variational equation of course
reflects the presence of the thin wave components $h_{j}$ in \eqref{2a}-\eqref{IC}. The
complications arise in the subsequent justification that the solutions of
said variational equation give rise to solutions of the resolvent equation
(in \eqref{a1} below) which are indeed in $D(\mathbf{A})$. In particular, we must
proceed delicately to show that the obtained thin layer solution components
of resolvent relation \eqref{a1}  satisfy the continuity conditions \eqref{2.5b}$_{2}$ and
\eqref{2.5b}$_{3}$.

Having established the existence of a $C_{0}$-semigroup of contractions $%
\left\{ e^{\mathbf{A}t}\right\} _{t\geq 0}\subset \mathcal{L}(\mathbf{H})$
which models the multilayer FSI PDE dynamics \eqref{2a}-\eqref{IC} , we will subsequently
show the strong decay of this semigroup; this is Theorem \ref{SS} below. Inasmuch
as our analysis of the regularizing effects of the resolvent operator $%
\mathcal{R}(\lambda ;\mathbf{A})$ is to be undertaken in future work --
assuming there be such underlying smoothness, at least for some geometrical
configurations of the polygonal boundary segments; see Remark \ref{remark_3}
below -- the compactness of $D(\mathbf{A})$ is generally questionable.
Accordingly, in order to establish asymptotic decay of solutions to the FSI
PDE dynamics \eqref{2a}-\eqref{IC} , we will work to satisfy the conditions of the
wellknown  \cite{A-B}; see also \cite{L-P}. In particular, we will show below that 
$\sigma (\mathbf{A})\cap i\mathbb{R}=\emptyset $. (In our future work on
discerning uniform decay properties of solutions to the multilayered FSI
system \eqref{2a}-\eqref{IC}, the spectral information in Theorem \ref{SS} is also requisite; see
e.g., the resolvent criteria in \cite{huang} and \cite{tomilov}.) In showing
the nonpresence of $\sigma (\mathbf{A})$ on the imaginary axis --in
particular, to handle the continuous spectrum of $\mathbf{A}$-- we will
proceed in a manner somewhat analogous to what was undertaken in \cite{A-P2}
(in which another coupled PDE system, with the coupling accomplished across
a boundary interface, is analyzed with a view towards stability). However,
the thin layer wave equation in \eqref{2.5b} again gives rise to complications: In
the course of eliminating the possibility of approximate spectrum of $%
\mathbf{A}$ on $i\mathbb{R}$, we find it necessary to invoke the wave
multipliers which are used in PDE control theory for \emph{uniform}
stabilization of boundary controlled waves: namely, inasmuch as each $h_{j}$%
-wave equation in \eqref{2.5b} carries the difference of the 3-D wave and heat fluxes
as a forcing term, we cannot immediately control the thick wave trace $%
\left. \frac{\partial w}{\partial \nu }\right\vert _{\Gamma _{s}}$ in $H^{-%
\frac{1}{2}}(\Gamma _{s})$-norm, this control being needed for strong decay.
(This issue absolutely does not appear for the previously considered 3-D heat-3-D wave FSI models of  \cite{Duyckaerts} and the other mentioned works, since therein we
have only the difference of heat and wave fluxes as a coupling boundary
condition, which immediately leads to a decent $H^{-\frac{1}{2}}(\Gamma _{s})
$ estimate of the wave normal derivative, owing to the thermal dissipation.)
Consequently, we must invoke static versions of the wave identities in [14],
 \cite{trigg} and \cite{AG1}, by way of estimating the normal derivative of (a component of)
the 3-D wave solution variable $w$ in \eqref{2d}; see relation (74) below.

\subsection{Notation}

For the remainder of the text norms $||\cdot||$ are taken to be $L^2(D)$ for the domain $D$.  Inner products in $L^2(D)$ is written $(\cdot,\cdot)$, while inner products  $L_2(\partial D)$ are written $\langle\cdot,\cdot\rangle$. The space $ H^s(D)$ will denote the Sobolev space of order $s$, defined on a domain $D$, and $H^s_0(D)$ denotes the closure of $C_0^{\infty}(D)$ in the $H^s(D)$ norm
which we denote by $\|\cdot\|_{H^s(D)}$ or $\|\cdot\|_{s,D}$. We make use of the standard notation for the trace of functions defined on a Lipschitz domain $D$, i.e. for a scalar function $\phi \in H^1(D)$,  we denote $\gamma(w)$ to be the trace mapping from $H^1(D)$ to $H^{1/2}(\partial D)$. We will also denote pertinent duality pairings as $(\cdot, \cdot)_{X \times X'}$. 

\section{Main Results}

\subsection{The thick wave-thin wave-heat Generator}

With respect to the above setting, the PDE system given in (\ref{2a})-(\ref{IC}) can be recast as an ODE in Hilbert space $\mathbf{H}$. That
is, if $\Phi (t)=\left[ u,h_{1},\frac{\partial }{\partial t}h_{1},...,h_{K},\frac{%
\partial }{\partial t}h_{K},w,w_{t}\right] \in C([0,T];\mathbf{H})$ solves (%
\ref{2a})-(\ref{IC}) for $\Phi _{0}\in \mathbf{H}$, then there is a modeling operator $\mathbf{A}:D(%
\mathbf{A})\subset \mathbf{H}\rightarrow \mathbf{H}$ such that $\Phi(\cdot) $
satisfies,%
\begin{equation}
\frac{d}{dt}\Phi (t)=\mathbf{A}\Phi (t)\text{; \ }\Phi (0)=\Phi _{0}.
\label{ODE}
\end{equation}%
In fact, this operator $\mathbf{A}:D(\mathbf{A})\subset \mathbf{H}%
\rightarrow \mathbf{H}$ is defined as follows:%
\begin{equation}
\mathbf{A}=\left[ 
\begin{array}{cccccccc}
\Delta & 0 & 0 & 0 & 0 & 0 & 0 & 0 \\ 
0 & 0 & I & \cdots & 0 & 0 & 0 & 0 \\ 
-\frac{\partial }{\partial \nu }|_{\Gamma _{1}} & (\Delta -I) & 0 & \cdots & 
0 & 0 & \frac{\partial }{\partial \nu }|_{\Gamma _{1}} & 0 \\ 
\vdots & \vdots & \vdots & \cdots & \vdots & \vdots & \vdots & \vdots \\ 
0 & 0 & 0 & \cdots & 0 & I & 0 & 0 \\ 
-\frac{\partial }{\partial \nu }|_{\Gamma _{K}} & 0 & 0 & \cdots & (\Delta
-I) & 0 & \frac{\partial }{\partial \nu }|_{\Gamma _{K}} & 0 \\ 
0 & 0 & 0 & \cdots & 0 & 0 & 0 & I \\ 
0 & 0 & 0 & \cdots & 0 & 0 & \Delta & 0%
\end{array}%
\right] ;  \label{4a}
\end{equation}%
\begin{equation}
\begin{array}{l}
D(\mathbf{A})=\left\{ \left[ u_{0},h_{01},h_{11},\ldots
,h_{0K},h_{1K},w_{0},w_{1}\right] \in \mathbf{H}:\right. \\ 
\text{ \ \ \textbf{(A.i)} }u_{0}\in H^{1}(\Omega _{f})\text{, }h_{1j}\in H^{1}(\Gamma
_{j})\text{ for }1\leq j\leq K\text{, }w_{1}\in H^{1}(\Omega _{s})\text{;}
\\ 
\text{ \ }\left. \text{\textbf{(A.ii)} (a) }\Delta u_{0}\in L^{2}(\Omega _{f})\text{, 
}\Delta w_{0}\in L^{2}(\Omega _{s})\text{, (b) }\Delta h_{0j}-\frac{\partial
u_{0}}{\partial \nu }|_{\Gamma _{j}}+\frac{\partial w_{0}}{\partial \nu }%
|_{\Gamma _{j}}\in L^{2}(\Gamma _{j})\text{ \ for \ } 1\leq j\leq K\text{;}%
\right. \\ 
\text{ \ \ \ \ \ \ (c) }\left. \dfrac{\partial h_{0j}}{\partial n_{j}}%
\right\vert _{\partial \Gamma _{j}}\in H^{-\frac{1}{2}}(\partial \Gamma _{j})%
\text{, \ for \ } 1\leq j\leq K\text{;} \\ 
\text{ \ }\left. \text{\textbf{(A.iii)} }u_{0}|_{\Gamma _{f}}=0,\ \ u_{0}|_{\Gamma
_{j}}=h_{1j}=w_{1}|_{\Gamma _{j}},\ \text{for }1\leq j\leq K\text{;}\right.
\\ 
\text{ \ }\left. \text{\textbf{(A.iv)} For }1\leq j\leq K\text{: }\right. \\ 
\text{ \ \ \ \ \ \ (a) }h_{1j}|_{\partial \Gamma _{j}\cap \partial \Gamma
_{l}}=h_{1l}|_{\partial \Gamma _{j}\cap \partial \Gamma _{l}}\text{ on \ }%
\partial \Gamma _{j}\cap \partial \Gamma _{l}\text{, for all }1\leq l\leq K%
\text{ such that }\partial \Gamma _{j}\cap \partial \Gamma _{l}\neq
\emptyset ; \\ 
\text{ \ \ \ \ \ \ \ }\left. \text{(b) }\left. \dfrac{\partial h_{0j}}{%
\partial n_{j}}\right\vert _{\partial \Gamma _{j}\cap \partial \Gamma
_{l}}=-\left. \dfrac{\partial h_{0_{l}}}{\partial n_{l}}\right\vert
_{\partial \Gamma _{j}\cap \partial \Gamma _{l}}\text{ on \ }\partial \Gamma
_{j}\cap \partial \Gamma _{l}\text{, for all }1\leq l\leq K\text{ such that }%
\partial \Gamma _{j}\cap \partial \Gamma _{l}\neq \emptyset \right\} .%
\end{array}
\label{dom}
\end{equation}%

Now, in our first result, we provide a semigroup wellposedness for $\mathbf{A}:D(\mathbf{A})\subset \mathbf{H}%
\rightarrow \mathbf{H}$. This is given in the following theorem:

\begin{theorem}
\label{well}The operator $\mathbf{A}:D(\mathbf{A})\subset \mathbf{H}%
\rightarrow \mathbf{H}$, defined in (\ref{4a})-(\ref{dom}), generates a $%
C_{0}$-semigroup of contractions. Consequently, the solution $\Phi (t)=\left[
u,h_{1},\frac{\partial }{\partial t}h_{1},...,h_{K},\frac{\partial }{%
\partial t}h_{K},w,w_{t}\right] $ of (\ref{2a})-(\ref{IC}), or equivalently (%
\ref{ODE}), is given by 
\begin{equation*}
\Phi (t)=e^{\mathbf{A}t}\Phi _{0}\in C([0,T];\mathbf{H})\text{,}
\end{equation*}%
where $\Phi _{0}=\left[ u_{0},h_{01},h_{11},...,h_{0K},h_{1K},w_{0},w_{1}%
\right] \in \mathbf{H}$.
\end{theorem}
After proving the existence and uniqueness of the solution, in our second result, we investigate the long term analysis of this solution. Our main goal here is to show that the solution to the system (\ref{2a})-(\ref{IC}) is strongly stable, which is given as follows: 

\begin{theorem}
\label{SS} For the modeling generator $\mathbf{A}:D(\mathbf{A})\subset \mathbf{H}%
\rightarrow \mathbf{H}$ of (\ref{2a})-(\ref{IC}), one has $ \sigma(\mathbf{A})\cap i\mathbb{R} $. Consequently, the $C_{0}-$semigroup $\left\{ e^{\mathbf{A}t}\right\} _{t\geq 0}$given in Theorem \ref{well} is
strongly stable. That is, the solution $\Phi (t)$ of the PDE (\ref{2a})-(\ref%
{IC}) tends asymptotically to the zero state for all initial data $\Phi
_{0}\in \mathbf{H.}$
\end{theorem}

\begin{remark}
\label{remark_1}The wellposedness and stability statements Theorems 1 and 2
are equally valid in the lower dimensional setting $n=2$; i.e., for
multilayered 2D heat -- 1D wave -- 2D wave coupled PDE systems (2)-(5), in
which interface $\Gamma _{s}$ is the boundary of a convex polygonal domain $%
\Omega _{s}$ (and so each segment $\Gamma _{j}$ is a line segment). (Also,
analogous to the present 3D setting, $\Omega _{f}$ is a Lipschitz domain
with $\partial \Omega _{f}=\Gamma _{s}\cup \Gamma _{f}$, with $\overline{%
\Gamma }_{s}\cap \overline{\Gamma }_{f}=\emptyset $.
\end{remark}

\begin{remark}
\label{remark_2}Inasmuch as we wish in future to turn our attention to the
numerical analysis and simulation of solutions of the multilayered PDE
system (2)-(5), the boundary interface is taken here to be polyhedral, with
each polygonal boundary segment $\Gamma _{j}$ having its own wave equation
IC-BVP in variable $h_{j}$. Alternatively, the Theorems 1 and 2 will also
hold true in the case that boundary interface $\Gamma _{s}$ is smooth: in
this case, the \textquotedblleft thin\textquotedblright\ wave equation -- in
solution variable $h$, say -- will have its spatial displacements described
by the Laplace Beltrami operator $\Delta ^{\prime }$. That is, for the
multilayered FSI model on a smooth boundary interface $\Gamma _{s}$, the
thin wave PDE component in (3) is replaced with 
\[
h_{tt}-\Delta ^{\prime }h+h=\left. \frac{\partial w}{\partial \nu }%
\right\vert _{\Gamma _{s}}-\left. \frac{\partial u}{\partial \nu }%
\right\vert _{\Gamma _{s}}\text{ \ on }(0,T)\times \Gamma _{s}\text{, }
\]%
with the matching velocity B.C.'s%
\[
\left. w_{t}\right\vert _{\Gamma _{s}}=h_{t}=\left. u\right\vert _{\Gamma
_{s}}\text{ \ on }(0,T)\times \Gamma _{s}.
\]%
The heat and thick wave PDE components in (2) and (4) respectively are
unchanged. In addition, there are the initial conditions 
\[
\lbrack u(0),h(0),h_{t}(0),w(0),w_{t}(0)]=[u_{0},h_{0},h_{1},w_{0},w_{1}]\in
L^{2}(\Omega _{f})\times H^{1}(\Gamma _{s})\times L^{2}(\Gamma _{s})\times
H^{1}(\Omega _{s})\times L^{2}(\Omega _{s}).
\]%
Also, the initial conditions satisfy the compatibility conditions $\left.
w_{0}\right\vert _{\Gamma _{s}}=h_{0}$.
\end{remark}

\begin{remark}
\label{remark_3}In line with what is observed in [31] and [32], it seems
possible -- at least for certain configurations of the polygonal segments $%
\Gamma _{j}$, $j=1,...K$ -- that the domain $D(\mathbf{A})$ of the
multilayer FSI generator (as prescribed in (A.i)-(A.iv) above) manifests a
regularity higher than that of finite energy; i.e., $D(\mathbf{A})\subset
H^{1}(\Omega _{f})\times H^{1+\rho _{1}}(\Gamma _{1})\times H^{1}(\Gamma
_{1})\times ...\times H^{1+\rho _{1}}(\Gamma _{K})\times H^{1}(\Gamma
_{K})\times H^{1+\rho _{2}}(\Omega _{s})\times H^{1}(\Omega _{s})$, where
parameters $\rho _{1}$,$\rho _{2}>0$. In the course of our future work --
e.g., an analysis of uniform decay properties of the FSI model (2)-(5) --
this higher regularity will be fleshed out. We should note that in the case
of a smooth boundary interface $\Gamma _{s}$ (see Remark \ref{remark_2}), smoothness of
the associated FSI semigroup generator domain comes directly from classic
elliptic regularity. In dimension $n=2$ (see Remark \ref{remark_1}), smoothness of the
semigroup generator domain can be inferred by the work of P. Grisvard; see
e.g., \cite{Grisvard2}, Theorem 2.4.3 of p. 57, along with Remarks 2.4.5 and
2.4.6 therein.
\end{remark}

\section{Wellposedness--\textit{Proof of Theorem \ref{well}}}

This section is devoted to prove the Hadamard well-posedness of the coupled system given in (\ref{2a})-(\ref{IC}). Our proof hinges on the application of the Lumer Phillips Theorem which assures the existence of a $C_{0}$-semigroup of contractions $\left\{ e^{\mathbf{A}%
t}\right\} _{t\geq 0}$ once we establish that $\mathbf{A}$ is maximal
dissipative.\\

\noindent \textbf{\textit{Proof of Theorem \ref{well}:}} In order to prove the maximal
dissipativity of $\mathbf{A}$, we will follow a few steps: \\

\noindent\emph{\textbf{Step 1 (Dissipativity of }}$\mathbf{A}$\emph{) } Given data $\Phi
_{0}$ in (\ref{stat}) to be in $D(\mathbf{A})$,

\begin{eqnarray}
(\mathbf{A}\Phi_0 ,\Phi_0 )_{\mathbf{H}} &=&(\Delta u_{0},u_{0})_{\Omega
_{f}}+\sum\limits_{j=1}^{K}(\nabla h_{1j},\nabla h_{0j})_{\Gamma _{j}} 
\notag \\
&&+\sum\limits_{j=1}^{K}(h_{1j},h_{0j})_{\Gamma
_{j}}+\sum\limits_{j=1}^{K}([\Delta -I]h_{0j},h_{1j})_{\Gamma _{j}}  \notag
\\
&&+\sum\limits_{j=1}^{K}\left\langle \frac{\partial w_{0}}{\partial \nu }%
,h_{1j}\right\rangle _{\Gamma _{j}}-\sum\limits_{j=1}^{K}\left\langle \frac{%
\partial u_{0}}{\partial \nu },h_{1j}\right\rangle _{\Gamma _{j}}  \notag \\
&&+(\nabla w_{1},\nabla w_{0})_{\Omega _{s}}+(\Delta w_{0},w_{1})_{\Omega
_{s}}  \notag \\
&=&-(\nabla u_{0},\nabla u_{0})_{\Omega _{f}}+\left\langle \frac{\partial }{%
\partial \nu }u_{0},u_{0}\right\rangle _{\Gamma _{s}}  \notag \\
&&+\sum\limits_{j=1}^{K}(\nabla h_{1j},\nabla h_{0j})_{\Gamma
_{j}}+\sum\limits_{j=1}^{K}(h_{1j},h_{0j})_{\Gamma _{j}}  \notag \\
&&-\sum\limits_{j=1}^{K}\overline{(\nabla h_{1j},\nabla h_{0j})}_{\Gamma
_{j}}-\sum\limits_{j=1}^{K}\overline{(h_{1j},h_{0j})}_{\Gamma
_{j}}+\sum\limits_{j=1}^{K}\left( \frac{\partial h_{0j}}{\partial n_{j}}%
,h_{1j}\right) _{\partial \Gamma _{j}}  \notag \\
&&+\sum\limits_{j=1}^{k}(\frac{\partial w_{0}}{\partial \nu }%
,h_{1j})_{\Gamma _{j}}-\sum\limits_{j=1}^{k}\left\langle \frac{\partial
u_{0}}{\partial \nu },h_{1j}\right\rangle _{\Gamma _{j}}  \notag \\
&&+(\nabla w_{1},\nabla w_{0})_{\Omega _{s}}-\overline{(\nabla w_{1},\nabla
w_{0})}_{\Omega _{s}}-\left\langle \frac{\partial w_{0}}{\partial \nu }%
,w_{1}\right\rangle _{\Gamma _{s}}.  \label{ten}
\end{eqnarray}%
(In the last expression, we are implicitly using the fact the unit normal
vector $\nu $ is \emph{interior }with respect to $\Gamma _{s}$.) Note now
via domain criterion (A.iv),we have for fixed index $j$, $1\leq j\leq K$, 
\begin{equation*}
\left( \frac{\partial h_{0j}}{\partial n_{j}},h_{1j}\right) _{\partial
\Gamma _{j}}=\sum\limits_{\substack{ 1\leq l\leq K \\ \partial \Gamma
_{j}\cap \partial \Gamma _{l}\neq \emptyset }}-\left( \frac{\partial h_{0l}}{%
\partial n_{l}},h_{1l}\right) _{\partial \Gamma _{j}\cap \partial \Gamma
_{l}}.
\end{equation*}%
Such relation gives then the inference 
\begin{equation}
\sum\limits_{j=1}^{K}\left( \frac{\partial h_{0j}}{\partial n_{j}}%
,h_{1j}\right) _{\partial \Gamma _{j}}=0.  \label{10.5}
\end{equation}%
Applying this relation and domain criterion (A.iii) to (\ref{ten}), we then
have%
\begin{equation}
\begin{array}{l}
(\mathbf{A}\Phi_0,\Phi_0 )_{\mathbf{H}}=-||\nabla u_{0}||_{_{\Omega _{f}}}^{2}
\\ 
\text{ \ \ }+2i\sum\limits_{j=1}^{K} \text{Im}(\nabla h_{1j},\nabla
h_{0j})_{\Gamma _{j}}+2i\sum\limits_{j=1}^{K}\text{Im}(h_{1j},h_{0j})_{%
\Gamma _{j}} \\ 
\text{ \ \ \ }+2i\text{Im}(\nabla w_{1},\nabla w_{0})_{\Omega _{s}},%
\end{array}
\label{dissi}
\end{equation}
which gives $$\text{Re}(\mathbf{A}\Phi ,\Phi )_{\mathbf{H}}\leq 0.$$ \\
\noindent\emph{\textbf{Step 2 (The Maximality of} }$\mathbf{A}$\emph{)}  Given parameter $%
\lambda >0$, suppose $\Phi =\left[ u_{0},h_{01},h_{11},\ldots
,h_{0K},h_{1K},w_{0},w_{1}\right] \in D(\mathbf{A})$ is a solution of the
equation%
\begin{equation}
(\lambda I -\mathbf{A})\Phi =\Phi ^{\ast },  \label{a1}
\end{equation}%
where $\Phi ^{\ast }=\left[ u_{0}^{\ast },h_{01}^{\ast },h_{11}^{\ast
},\ldots ,h_{0K}^{\ast },h_{1K}^{\ast },w_{0}^{\ast },w_{1}^{\ast }\right]
\in \mathbf{H}$. Then in PDE terms, the abstract equation (\ref{a1}) becomes%
\begin{equation}
\left\{ 
\begin{array}{l}
\lambda u_{0}-\Delta u_{0} =u_{0}^{\ast }\text{ \ \ in \ \ \ }\Omega _{f} \\
u_{0}|_{\Gamma _{f}} =0\text{ \ \ on \ \ }\Gamma _{f};  
\end{array}
\right.  \label{p1}
\end{equation}%
and for $1\leq j\leq K,$%
\begin{equation}
\left\{ 
\begin{array}{l}
\lambda h_{0j}-h_{1j}=h_{0j}^{\ast }\text{ \ \ in \ \ }\Gamma _{j} \\ 
\lambda h_{1j}-\Delta h_{0j}+h_{0j}-\dfrac{\partial w_{0}}{\partial \nu }+%
\dfrac{\partial u_{0}}{\partial \nu }=h_{1j}^{\ast }\text{\ \ \ in \ \ }%
\Gamma _{j} \\ 
u_{0}|_{\Gamma _{j}}=h_{1j}=w_{1}|_{\Gamma _{j}}\text{\ \ \ in \ \ }\Gamma
_{j} \\ 
h_{0j}|_{\partial \Gamma _{j}\cap \partial \Gamma _{l}}=h_{0l}|_{\partial
\Gamma _{j}\cap \partial \Gamma _{l}}\text{ on \ }\partial \Gamma _{j}\cap
\partial \Gamma _{l}\text{, for all }1\leq l\leq K\text{ such that }\partial
\Gamma _{j}\cap \partial \Gamma _{l}\neq \emptyset \\  \label{p2}
\left. \dfrac{\partial h_{0j}}{\partial n_{j}}\right\vert _{\partial \Gamma
_{j}\cap \partial \Gamma _{l}}=-\left. \dfrac{\partial h_{0_{l}}}{\partial
n_{l}}\right\vert _{\partial \Gamma _{j}\cap \partial \Gamma _{l}}\text{ on
\ }\partial \Gamma _{j}\cap \partial \Gamma _{l}\text{, for all }1\leq l\leq
K\text{ such that }\partial \Gamma _{j}\cap \partial \Gamma _{l}\neq
\emptyset ;%
\end{array}
\right. 
\end{equation}
and also%
\begin{equation}
\left\{ 
\begin{array}{l}
\lambda w_{0}-w_{1} =w_{0}^{\ast }\text{ \ \ in \ \ \ }\Omega _{s}  
\\
\lambda w_{1}-\Delta w_{0} =w_{1}^{\ast }\text{ \ \ in \ \ \ }\Omega _{s}.
\end{array}
\right. \label{p3}
\end{equation}

\medskip

\noindent With respect to this static PDE system, we multiply the heat equation in (%
\ref{p1}) by test function $\varphi \in H_{\Gamma _{f}}^{1}(\Omega _{f})$,
where%
\begin{equation*}
H_{\Gamma _{f}}^{1}(\Omega _{f})=\left\{ \zeta \in H^{1}(\Omega _{f}):\zeta
|_{\Gamma _{f}}=0\right\} .  \label{4}
\end{equation*}%
Upon integrating and invoking Green's Theorem, then solution component $%
u_{0} $ satisfies the variational relation, 
\begin{equation}
\lambda (u_{0},\varphi )_{\Omega _{f}}+(\nabla u_{0},\nabla \varphi
)_{\Omega _{f}}-\left\langle \frac{\partial u_{0}}{\partial v},\varphi
\right\rangle _{\Gamma _{s}}=(u_{0}^{\ast },\varphi )_{\Omega _{f}}\text{ \
for }\varphi \in H_{\Gamma _{f}}^{1}(\Omega _{f})\text{.}  \label{5}
\end{equation}
In addition, define Hilbert space $\mathcal{V}$ by%
\begin{eqnarray}
\mathcal{V} &=&\left\{ \left[ \psi _{1},...,\psi _{K}\right] \in
H^{1}(\Gamma _{1})\times ...\times H^{1}(\Gamma _{K}):\right. \text{For\ all \ } 1\leq j\leq K, \notag \\
&&\left. \psi _{j}|_{\partial \Gamma _{j}\cap \partial \Gamma _{l}}=\psi
_{l}|_{\partial \Gamma _{j}\cap \partial \Gamma _{l}}\text{ on \ }\partial
\Gamma _{j}\cap \partial \Gamma _{l}\text{, for all }1\leq l\leq K\text{
such that }\partial \Gamma _{j}\cap \partial \Gamma _{l}\neq \emptyset
\right\}  \label{V}
\end{eqnarray}

\medskip

\noindent Therewith, we multiply both sides of the $h_{0j}$-wave equation in (\ref{p2}%
) by component $\psi _{j}$ of $\mathbf{\psi }\in \mathcal{V}$, for $1\leq
j\leq K$. Upon integration we have for $\mathbf{\psi }\in \mathcal{V}$,%
\begin{equation*}
\left[ 
\begin{array}{c}
\lambda (h_{11},\psi _{1})_{\Gamma _{1}}-(\Delta h_{01},\psi _{1})_{\Gamma
_{1}}+(h_{01},\psi _{1})_{\Gamma _{1}}-(\frac{\partial }{\partial \nu }%
w_{0},\psi _{1})_{\Gamma _{1}}+(\frac{\partial }{\partial \nu }u_{0},\psi
_{1})_{\Gamma _{1}} \\ 
\vdots \\ 
\lambda (h_{1K},\psi _{K})_{\Gamma _{K}}-(\Delta h_{0K},\psi _{K})_{\Gamma
_{K}}+(h_{0K},\psi _{K})_{\Gamma _{K}}-(\frac{\partial }{\partial \nu }%
w_{0},\psi _{K})_{\Gamma _{K}}+(\frac{\partial }{\partial \nu }u_{0},\psi
_{K})_{\Gamma _{K}}%
\end{array}%
\right] =%
\left[ 
\begin{array}{c}
(h_{11}^{\ast },\psi _{1})_{\Gamma _{1}} \\ 
\vdots \\ 
(h_{1K}^{\ast },\psi _{K})_{\Gamma _{K}}%
\end{array}%
\right] 
\end{equation*}

For each vector component, we subsequently integrate by parts while invoking
the resolvent relations in (\ref{p2}) (and using the domain criterion
(A.iv.b)). Summing up the components of the resulting vectors, we see that
the solution component $\left[ h_{11},...,h_{1K}\right] \in \mathcal{V}$ of (%
\ref{a1}) satisfies 
\begin{eqnarray}
&&\sum\limits_{j=1}^{K}\left[ \lambda (h_{1j},\psi _{j})_{\Gamma _{j}}+%
\frac{1}{\lambda }(\nabla h_{1j},\nabla \psi _{j})_{\Gamma _{j}}+\frac{1}{%
\lambda }(h_{1j},\psi _{j})_{\Gamma _{j}}+(\frac{\partial }{\partial \nu }%
u_{0}-\frac{\partial }{\partial \nu }w_{0},\psi _{j})_{\Gamma _{j}}\right] 
\notag \\
&&\text{ \ \ \ \ \ }=\sum\limits_{j=1}^{K}\left[ (h_{1j}^{\ast },\psi
_{j})_{\Gamma _{j}}-\frac{1}{\lambda }(h_{0j}^{\ast },\psi _{j})_{\Gamma
_{j}}-\frac{1}{\lambda }(\nabla h_{0j}^{\ast },\nabla \psi _{j})_{\Gamma
_{j}}\right] \text{, \ for }\mathbf{\psi }\in \mathcal{V}.  \label{6}
\end{eqnarray}

\medskip

\noindent Moreover, multiplying the both sides of the wave equation in (\ref{p3}) by $%
\xi \in H^{1}(\Omega _{s})$, and integrating by parts -- while using the
resolvent relations in (\ref{p3}) -- we see that the solution component $%
w_{1}$ of (\ref{a1}) satisfies%
\begin{equation}
\lambda (w_{1},\xi )_{\Omega _{s}}+\frac{1}{\lambda }(\nabla w_{1},\nabla
\xi )_{\Omega _{s}}+(\frac{\partial }{\partial \nu }w_{0},\xi )_{\Gamma
_{s}}=(w_{1}^{\ast },\xi )_{\Omega _{s}}-\frac{1}{\lambda }(\nabla
w_{0}^{\ast },\nabla \xi )_{\Omega _{s}}\text{, for }\xi \in H^{1}(\Omega
_{s}).  \label{7}
\end{equation}

\smallskip

\noindent Set now%
$$ \mathbf{W} \equiv \left\{ [\varphi ,\psi _{1},...,\psi _{K},\xi ]\in
H_{\Gamma _{f}}^{1}(\Omega _{f})\times \mathcal{V}\times H^{1}(\Omega
_{s}):\varphi |_{\Gamma _{j}}=\psi _{j}=\xi |_{\Gamma _{j}},\text{ for\ } 1\leq j\leq K%
\right\};$$ 
\begin{equation} \left\Vert \lbrack \varphi ,\psi _{1},...,\psi _{K},\xi ]\right\Vert _{%
\mathbf{W}}^{2} =\left\Vert \nabla \varphi \right\Vert _{\Omega
_{f}}^{2}+\sum\limits_{j=1}^{K}\left[ \left\Vert \nabla \psi
_{j}\right\Vert _{\Gamma _{j}}^{2}+\left\Vert \psi _{j}\right\Vert _{\Gamma
_{j}}^{2}\right] +\left\Vert \nabla \xi\right\Vert _{\Omega _{s}}^{2}.
\label{W} \end{equation}

\noindent With respect to this Hilbert space, we have the following conclusion, upon
adding (\ref{5}), (\ref{6}) and (\ref{7}): if $\Phi =\left[
u_{0},h_{01},h_{11},\ldots ,h_{0k},h_{1k},w_{0},w_{1}\right] \in D(\mathbf{A}%
)$ solves (\ref{a1}), then necessarily its solution components $\left[
u_{0},h_{11},\ldots ,h_{1K},w_{1}\right] \in \mathbf{W}$ satisfy \ for $%
\left[ \varphi ,\mathbf{\psi },\xi \right] \in \mathbf{W}$,%
\begin{equation}
\begin{array}{l}
\lambda (u_{0},\varphi )_{\Omega _{f}}+(\nabla u_{0},\nabla \varphi
)_{\Omega _{f}}+\lambda (w_{1},\xi )_{\Omega _{s}}+\frac{1}{\lambda }(\nabla
w_{1},\nabla \xi )_{\Omega _{s}} \\ 
\text{ \ }+\sum\limits_{j=1}^{K}\left[ \lambda (h_{1j},\psi _{j})_{\Gamma
_{j}}+\frac{1}{\lambda }(\nabla h_{1j},\nabla \psi _{j})_{\Gamma _{j}}+\frac{%
1}{\lambda }(h_{1j},\psi _{j})_{\Gamma _{j}}\right] =\mathbf{F}_{\lambda
}\left( \left[ 
\begin{array}{c}
\varphi \\ 
\mathbf{\psi } \\ 
\xi%
\end{array}%
\right] \right) ;%
\end{array}
\label{8}
\end{equation}
where 
\begin{equation}
\mathbf{F}_{\lambda }\left( \left[ 
\begin{array}{c}
\varphi \\ 
\mathbf{\psi } \\ 
\xi%
\end{array}%
\right] \right) =(u_{0}^{\ast },\varphi )_{\Omega
_{f}}+\sum\limits_{j=1}^{K}\left[ (h_{1j}^{\ast },\psi _{j})_{\Gamma _{K}}-%
\frac{1}{\lambda }(h_{0j}^{\ast },\psi _{j})_{\Gamma _{j}}-\frac{1}{\lambda }%
(\nabla h_{0j}^{\ast },\nabla \psi _{j})_{\Gamma _{j}}\right] +(w_{1}^{\ast
},\xi )_{\Omega _{s}}-\frac{1}{\lambda }(\nabla w_{0}^{\ast },\nabla \xi
)_{\Omega _{s}}.  \label{F}
\end{equation}

\medskip

\noindent In sum, in order to recover the solution $\Phi =\left[ u_{0},h_{01},h_{11},%
\ldots ,h_{0K},h_{1K},w_{0},w_{1}\right] \in D(\mathbf{A})$ to (\ref{a1}),
one can straightaway apply the Lax-Milgram Theorem to the operator $\mathbf{B%
}\in \mathcal{L}(\mathbf{W},\mathbf{W}^{\ast })$, given by

\begin{eqnarray}
&&\left\langle \mathbf{B}\left[ 
\begin{array}{c}
\varphi \\ 
\psi _{1} \\ 
\vdots \\ 
\psi _{k} \\ 
\xi%
\end{array}%
\right] ,\left[ 
\begin{array}{c}
\widetilde{\varphi } \\ 
\widetilde{\psi }_{1} \\ 
\vdots \\ 
\widetilde{\psi }_{k} \\ 
\widetilde{\xi }%
\end{array}%
\right] \right\rangle _{\mathbf{W}^{\ast }\times \mathbf{W}} = \lambda (\varphi ,\widetilde{\varphi })_{\Omega _{f}}+(\nabla \varphi
,\nabla \widetilde{\varphi })_{\Omega _{f}}+\lambda (\xi ,\widetilde{\xi }%
)_{\Omega _{s}}+\frac{1}{\lambda }(\nabla \xi ,\widetilde{\nabla \xi }%
)_{\Omega _{s}} \notag
 \end{eqnarray} 
 $$
+\sum\limits_{j=1}^{K}\left[ \lambda (\psi _{j},\widetilde{\psi }%
_{j})_{\Gamma _{j}}+\frac{1}{\lambda }(\nabla \psi _{j},\widetilde{\nabla
\psi }_{j})_{\Gamma _{j}}+\frac{1}{\lambda }(\psi _{j},\widetilde{\psi }%
_{j})_{\Gamma _{j}}\right] .  \label{8.5} $$
It is clear that $\mathbf{B}\in \mathcal{L}(\mathbf{W},\mathbf{W}^{\ast })$
is $\mathbf{W}$-elliptic; so by the Lax-Milgram Theorem, the equation (\ref%
{8}) has a unique solution 
\begin{equation}
\left[ u_{0},h_{11},\ldots ,h_{1K},w_{1}\right] \in \mathbf{W}.  \label{9}
\end{equation}%
Subsequently, we set%
\begin{equation}
\left\{ 
\begin{array}{l}
h_{0j}=\frac{h_{1j}+h_{0j}^{\ast }}{\lambda }\text{,\ for }1\leq j\leq K,%
\\ 
w_{0}=\frac{w_{1}+w_{0}^{\ast }}{\lambda }.%
\end{array}%
\right.  \label{10}
\end{equation}%
In particular, since the data $\left[ u_{0}^{\ast },h_{01}^{\ast
},h_{11}^{\ast },\ldots ,h_{0k}^{\ast },h_{1k}^{\ast },w_{0}^{\ast
},w_{1}^{\ast }\right] \in \mathbf{H},$ then the relations in (\ref{10})
give that 
\begin{equation}
w_{0}|_{\Gamma _{j}}=h_{0j},\text{ \ \ }1\leq j\leq K.  \label{18}
\end{equation}

\medskip

\noindent We further show that the dependent variable $\Phi =\left[
u_{0},h_{01},h_{11},\ldots ,h_{0k},h_{1k},w_{0},w_{1}\right] ,$ given by the
solution of (\ref{8}) and (\ref{10}), is an element of $D(\mathbf{A})$:
If we take $[\varphi ,0,\ldots ,0,0]\in \mathbf{W}$ in (\ref{8}), where $%
\varphi \in \mathcal{D}(\Omega _{f})$, then we have%
\begin{equation*}
\lambda (u_{0},\varphi )_{\Omega _{f}}-(\Delta u_{0},\varphi )_{\Omega
_{f}}=(u_{0}^{\ast },\varphi )_{\Omega _{f}}\text{\ \ \ }\forall \text{ }%
\varphi \in \mathcal{D}(\Omega _{f}),
\end{equation*}%
whence%
\begin{equation}
\lambda u_{0}-\Delta u_{0}=u_{0}^{\ast }\text{ in }L^{2}(\Omega _{f}).
\label{11}
\end{equation}%
Subsequently, the fact that $\left\{ \Delta u_{0},u_{0}\right\} \in
L^{2}(\Omega _{f})\times H^{1}(\Omega _{f})$ gives 
\begin{equation}
\frac{\partial u_{0}}{\partial v}|_{\Gamma _{s}}\in H^{-\frac{1}{2}}(\Gamma
_{s}).  \label{12}
\end{equation}%
In turn, using the relations in (\ref{10}), if we take $[0,0,\ldots ,0,\xi
]\in \mathbf{W}$, where $\xi \in \mathcal{D}(\Omega _{s})$, then upon
integrating by parts, we have%
\begin{equation*}
\lambda (w_{1},\xi )_{\Omega _{s}}-(\Delta w_{0},\xi )_{\Omega
_{s}}=(w_{1}^{\ast },\xi )_{\Omega _{s}}\text{ \ \ \ }\forall \text{ }\xi
\in \mathcal{D}(\Omega _{s}),
\end{equation*}%
and so 
\begin{equation}
\lambda w_{1}-\Delta w_{0}=w_{1}^{\ast }\text{ \ \ in \ }L^{2}(\Omega _{s}),
\label{13}
\end{equation}%
which gives that $\left\{ \Delta w_{0},w_{0}\right\} \in L^{2}(\Omega
_{s})\times H^{1}(\Omega _{s})$. A subsequent integration by parts yields
that 
\begin{equation}
\frac{\partial w_{0}}{\partial v}|_{\Gamma _{s}}\in H^{-\frac{1}{2}}(\Gamma
_{s}).  \label{14}
\end{equation}

\medskip

\noindent Moreover, let $\gamma _{s}^{+}\in \mathcal{L}(H^{\frac{1}{2}}(\Gamma
_{s}),H^{1}(\Omega _{s}))$ be the right continuous inverse for the Sobolev
trace map $\gamma _{s}\in \mathcal{L}(H^{1}(\Omega _{s}),H^{\frac{1}{2}}(\Gamma _{s}))$%
; viz.,%
\begin{equation*}
\gamma _{s}(f)=f|_{\Gamma _{s}}\text{ \ for }f\in C^{\infty }(\overline{%
\Omega _{s}}).
\end{equation*}%
Likewise, let $\gamma _{f}^{+}\in \mathcal{L}(H^{\frac{1}{2}}(\Gamma
_{s}),H_{\Gamma _{f}}^{1}(\Omega _{f}))$ denote the right inverse for the
Sobolev trace map \newline
$\gamma _{f}\in \mathcal{L}(H_{\Gamma _{f}}^{1}(\Omega _{f}),H^{\frac{1}{2}%
}(\Gamma _{s})).$ Also, for given $\psi _{j}\in H_{0}^{1}(\Gamma _{j}),$ \ $%
1\leq j\leq K,$ let%
\begin{equation}
\left( \psi _{j}\right) _{ext}(x)\equiv \left\{ 
\begin{array}{l}
\psi _{j},\text{ \ \ }x\in \Gamma _{j} \\ 
0,\text{ \ \ }x\in \Gamma _{s}\backslash \Gamma _{j}.%
\end{array}%
\right.  \label{15}
\end{equation}%
Then $\left( \psi _{j}\right) _{ext}\in H^{\frac{1}{2}}(\Gamma _{s})$ \ for
all $1\leq j\leq K$. We now specify test function $[\varphi ,\psi
_{1},...,\psi _{K},\xi ]\in \mathbf{W}$ in (\ref{8}): namely, $\psi _{j}\in
H_{0}^{1}(\Gamma _{j}),$ \ $1\leq j\leq K$, and 
\begin{equation}
\varphi \equiv \gamma _{f}^{+}\left[ \sum\limits_{j=1}^{K}\left( \psi
_{j}\right) _{ext}\right] ,\text{ \ \ \ \ \ \ \ \ }\xi \equiv \gamma _{s}^{+}%
\left[ \sum\limits_{j=1}^{K}\left( \psi _{j}\right) _{ext}\right] .
\label{16}
\end{equation}%
Therewith we have verbatim from (\ref{8}),%
\begin{eqnarray*}
&&\lambda (u_{0},\varphi )_{\Omega _{f}}+(\nabla u_{0},\nabla \varphi
)_{\Omega _{f}} \\
&&+\sum\limits_{j=1}^{K}\left[ \lambda (h_{1j},\psi _{j})_{\Gamma _{j}}+%
\frac{1}{\lambda }(\nabla h_{1j},\nabla \psi _{j})_{\Gamma _{j}}+\frac{1}{%
\lambda }(h_{1j},\psi _{j})_{\Gamma _{j}}\right] \\
&&+\lambda (w_{1},\xi )_{\Omega _{s}}+\frac{1}{\lambda }(\nabla w_{1},\nabla
\xi )_{\Omega _{s}} \\
&=&(u_{0}^{\ast },\varphi )_{\Omega _{f}}+\sum\limits_{j=1}^{k}\left[
(h_{1j}^{\ast },\psi _{j})_{\Gamma _{j}}-\frac{1}{\lambda }(\nabla
h_{0j}^{\ast },\nabla \psi _{j})_{\Gamma _{j}}-\frac{1}{\lambda }%
(h_{0j}^{\ast },\psi _{j})_{\Gamma _{j}}\right] \\
&&+(w_{1}^{\ast },\xi )_{\Omega _{s}}-\frac{1}{\lambda }(\nabla w_{0}^{\ast
},\nabla \xi )_{\Omega _{s}}.
\end{eqnarray*}%
Upon integrating by parts, and invoking the relations in (\ref{10}), as well
as (\ref{11})-(\ref{14}), we get%
\begin{equation}
\left\langle \frac{\partial u_{0}}{\partial \nu },\varphi \right\rangle
_{\Gamma _{s}}+\sum\limits_{j=1}^{K}\left[ \lambda (h_{1j},\psi
_{j})_{\Gamma _{j}}-(\Delta h_{0j},\psi _{j})_{\Gamma _{j}}+(h_{0j},\psi
_{j})_{\Gamma _{j}}\right] -\left\langle \frac{\partial w_{0}}{\partial \nu }%
,\xi \right\rangle _{\Gamma _{s}}=\sum\limits_{j=1}^{K}(h_{1j}^{\ast },\psi
_{j})_{\Gamma _{j}}.  \label{inter}
\end{equation}%
Since each test function component $\psi _{j}\in H_{0}^{1}(\Gamma _{j})$ is
arbitrary, we then deduce from this relation and (\ref{15})-(\ref{16}) that
each $h_{0j}$ solves%
\begin{equation}
\lambda h_{1j}-\Delta h_{0j}+h_{0j}-\frac{\partial w_{0}}{\partial \nu }+%
\frac{\partial u_{0}}{\partial \nu }=h_{1j}^{\ast }\text{ \ \ in \ }\Gamma
_{j},\text{ \ \ }1\leq j\leq K.  \label{17}
\end{equation}

\medskip

\noindent In addition, we have from (\ref{17}), (\ref{9}), (\ref{12}), and (\ref{14})
that $\left\{ \Delta h_{0j},h_{0j}\right\} \in \lbrack H^{1}(\Gamma
_{j})]^{\prime }\times H^{1}(\Gamma _{j})$, for $1\leq j\leq K$.
Consequently, an integration by parts gives that 
\begin{equation}
\frac{\partial h_{0j}}{\partial n_{j}}\in H^{-\frac{1}{2}}(\partial \Gamma
_{j})\text{, \ for }1\leq j\leq K.  \label{19}
\end{equation}

\medskip

\noindent Finally: Let given indices $j^{\ast },l^{\ast }$, $1\leq j^{\ast },l^{\ast
}\leq K$, satisfy $\partial \Gamma _{j^{\ast }}\cap $ $\partial \Gamma
_{l^{\ast }}\neq \emptyset $. Let $g$ be a given element in $H_{0}^{\frac{1}{%
2}+\epsilon }(\partial \Gamma _{j^{\ast }}\cap $ $\partial \Gamma _{l^{\ast
}})$. Then one has that $\tilde{g}_{j^{\ast }}\in H^{\frac{1}{2}+\epsilon
}(\partial \Gamma _{j^{\ast }})$ and $\tilde{g}_{l^{\ast }}\in H^{\frac{1}{2}%
+\epsilon }(\partial \Gamma _{l^{\ast }})$, where%
\begin{equation*}
\tilde{g}_{j^{\ast }}(x)\equiv\left\{ 
\begin{array}{l}
g(x)\text{, }x\in \partial \Gamma _{j^{\ast }}\cap \partial \Gamma _{l^{\ast
}} \\ 
0\text{, }x\in \partial \Gamma _{j^{\ast }}\backslash \left( \partial \Gamma
_{j^{\ast }}\cap \partial \Gamma _{l^{\ast }}\right) ;%
\end{array}%
\right. \text{ \ \ }\tilde{g}_{l^{\ast }}(x)\equiv\left\{ 
\begin{array}{l}
g(x)\text{, }x\in \partial \Gamma _{j^{\ast }}\cap \partial \Gamma _{l^{\ast
}} \\ 
0\text{, }x\in \partial \Gamma _{l^{\ast }}\backslash \left( \partial \Gamma
_{j^{\ast }}\cap \partial \Gamma _{l^{\ast }}\right)%
\end{array}%
\right.
\end{equation*}%
(see e.g., Theorem 3.33, p. 95 of \cite{mclean}). Subsequently, by the
(limited) surjectivity of the Sobolev Trace Map on Lipschitz domains-- see
e.g., Theorem 3.38, p.102 of \cite{mclean} -- there exists $\psi _{j^{\ast
}}\in H^{1+\epsilon }(\Gamma _{j^{\ast }})$ and $\psi _{l^{\ast }}\in
H^{1+\epsilon }(\Gamma _{l^{\ast }})$ such that 
\begin{equation}
\left. \psi _{j^{\ast }}\right\vert _{\partial \Gamma _{j^{\ast }}}=\tilde{g}%
_{j^{\ast }}\text{, \ and \ }\left. \psi _{l^{\ast }}\right\vert _{\partial
\Gamma _{l^{\ast }}}=\tilde{g}_{l^{\ast }}.  \label{g}
\end{equation}%
In turn, by the Sobolev Embedding Theorem, if we define, on $\overline{%
\Gamma }_{s}$ the function 
\begin{equation}
\Upsilon (x)\equiv \left\{ 
\begin{array}{l}
\psi _{j^{\ast }}(x)\text{, \ for }x\in \overline{\Gamma} _{j^{\ast }} \\ 
\psi _{l^{\ast }}(x)\text{, \ for }x\in \overline{\Gamma }_{l^{\ast }} \\ 
0\text{, \ for }x\in \overline{\Gamma }_{s}\backslash \left( \overline{\Gamma}
_{j^{\ast }}\cup \overline{\Gamma} _{l^{\ast }}\right) ,%
\end{array}%
\right.  \label{sig}
\end{equation}%
then $\Upsilon (x)\in C(\overline{\Gamma }_{s})$. Since also $\psi _{j^{\ast
}}\in H^{1}(\Gamma _{j^{\ast }})$ and $\psi _{l^{\ast }}\in H^{1}(\Gamma
_{l^{\ast }})$, we eventually deduce via an integration by parts that $%
\Upsilon \in H^{1}(\Gamma _{s})$. (See e.g., the proof of Theorem 2, p. 36
of \cite{ciarlet}.) With this $H^{1}$-function in hand, and with aforesaid
continuous right inverses $\gamma _{s}^{+}\in \mathcal{L}(H^{\frac{1}{2}%
}(\Gamma _{s}),H^{1}(\Omega _{s}))$ and $\gamma _{f}^{+}\in \mathcal{L}(H^{%
\frac{1}{2}}(\Gamma _{s}),H_{\Gamma _{f}}^{1}(\Omega _{f}))$, we specify the
vector%
\begin{equation}
\left[ \varphi ,\mathbf{\psi },\xi \right] \equiv\left[ \gamma _{f}^{+}(\Upsilon
),0,...,\psi _{j^{\ast }},0,...0,\psi _{l^{\ast }},...,0,\gamma
_{s}^{+}(\Upsilon )\right] \in \mathbf{W}\text{,}  \label{vec}
\end{equation}%
where again, space $\mathbf{W}$ is given in (\ref{W}). With this vector in
hand, we consider the thin wave equation in (\ref{17}): With respect to the
two fixed indices $1\leq j^{\ast },l^{\ast }\leq K$, we have via (\ref{17})
\begin{equation*}
\begin{array}{l}
\lambda \left( h_{1j^{\ast }},\psi _{j^{\ast }}\right) _{\Gamma _{j^{\ast
}}}-\left( \Delta h_{0j^{\ast }},\psi _{j^{\ast }}\right) _{\Gamma _{j^{\ast
}}}+\left( h_{0j^{\ast }},\psi _{j^{\ast }}\right) _{\Gamma _{j^{\ast
}}} \\ -\left( \frac{\partial w_{0}}{\partial \nu }-\frac{\partial u_{0}}{%
\partial \nu },\psi _{j^{\ast }}\right) _{\Gamma _{j^{\ast }}} +\lambda \left( h_{1l^{\ast }},\psi _{l^{\ast }}\right) _{\Gamma _{l^{\ast
}}}-\left( \Delta h_{0l^{\ast }},\psi _{l^{\ast }}\right) _{\Gamma _{l^{\ast
}}} \\ +\left( h_{0l^{\ast }},\psi _{l^{\ast }}\right) _{\Gamma _{l^{\ast
}}}-\left( \frac{\partial w_{0}}{\partial \nu }-\frac{\partial u_{0}}{%
\partial \nu },\psi _{l^{\ast }}\right) _{\Gamma _{l^{\ast }}}=\left(
h_{1j^{\ast }}^{\ast },\psi _{j^{\ast }}\right) _{\Gamma _{j^{\ast
}}}+\left( h_{1l^{\ast }}^{\ast },\psi _{l^{\ast }}\right) _{\Gamma
_{l^{\ast }}}.%
\end{array}%
\end{equation*}%
A subsequent integration by parts, with (\ref{vec}) in mind, subsequently
yields%
\begin{equation*}
\begin{array}{l}
\lambda \left( h_{1j^{\ast }},\psi _{j^{\ast }}\right) _{\Gamma _{j^{\ast
}}}+\left( \nabla h_{0j^{\ast }},\nabla \psi _{j^{\ast }}\right) _{\Gamma
_{j^{\ast }}}-\left\langle \frac{\partial h_{0j^{\ast }}}{\partial
n_{j^{\ast }}},g\right\rangle _{\partial \Gamma _{j^{\ast }}\cap \partial
\Gamma _{l^{\ast }}} +\left( h_{0j^{\ast }}, \psi _{j^{\ast }}\right) _{\Gamma
_{j^{\ast }}}\\ 
\text{ \ }+\lambda \left( h_{1l^{\ast }},\psi _{l^{\ast }}\right) _{\Gamma
_{l^{\ast }}}+\left( \nabla h_{0l^{\ast }},\nabla \psi _{l^{\ast }}\right)
_{\Gamma _{l^{\ast }}}-\left\langle \frac{\partial h_{0l^{\ast }}}{\partial
n_{l^{\ast }}},g\right\rangle _{\partial \Gamma _{j^{\ast }}\cap \partial
\Gamma _{l^{\ast }}}+\left( h_{0l^{\ast }},\psi _{l^{\ast }}\right)
_{\Gamma _{l^{\ast }}} \\ 
\text{ \ }+\left( \nabla w_{0},\nabla \xi \right) _{\Omega _{s}}+\left(
\Delta w_{0},\xi \right) _{\Omega _{s}}+\left( \nabla u_{0},\nabla \varphi
\right) _{\Omega _{f}}+\left( \Delta u_{0},\varphi \right) _{\Omega
_{f}}=\left( h_{1j^{\ast }}^{\ast },\psi _{j^{\ast }}\right) _{\Gamma
_{j^{\ast }}}+\left( h_{1l^{\ast }}^{\ast },\psi _{l^{\ast }}\right)
_{\Gamma _{l^{\ast }}}.%
\end{array}%
\end{equation*}%
Invoking (\ref{11}) and (\ref{13}), we then have%
\begin{equation*}
\begin{array}{l}
-\left\langle \frac{\partial h_{0j^{\ast }}}{\partial n_{j^{\ast }}}%
,g\right\rangle _{\partial \Gamma _{j^{\ast }}\cap \partial \Gamma _{l^{\ast
}}}-\left\langle \frac{\partial h_{0l^{\ast }}}{\partial n_{l^{\ast }}}%
,g\right\rangle _{\partial \Gamma _{j^{\ast }}\cap \partial \Gamma _{l^{\ast
}}} +\left( h_{0j^{\ast }}, \psi _{j^{\ast }}\right) _{\Gamma
_{j^{\ast }}}+\left( h_{0l^{\ast }},\psi _{l^{\ast }}\right)
_{\Gamma _{l^{\ast }}}\\ 
+\lambda \left( h_{1j^{\ast }},\psi _{j^{\ast }}\right) _{\Gamma _{j^{\ast
}}}+\left( \nabla h_{0j^{\ast }},\nabla \psi _{j^{\ast }}\right) _{\Gamma
_{j^{\ast }}}+\lambda \left( h_{1l^{\ast }},\psi _{l^{\ast }}\right)
_{\Gamma _{l^{\ast }}}+\left( \nabla h_{0l^{\ast }},\nabla \psi _{l^{\ast
}}\right) _{\Gamma _{l^{\ast }}} \\ 
+\left( \nabla w_{0},\nabla \xi \right) _{\Omega _{s}}+\lambda \left(
w_{1},\xi \right) _{\Omega _{s}}-\left( w_{1}^{\ast },\xi \right) _{\Omega
_{s}}+\left( \nabla u_{0},\nabla \varphi \right) _{\Omega _{f}}+\lambda
\left( u_{0},\varphi \right) _{\Omega _{f}}-\left( u_{0}^{\ast },\varphi
\right) _{\Omega _{f}} \\ 
=\left( h_{1j^{\ast }}^{\ast },\psi _{j^{\ast }}\right) _{\Gamma _{j^{\ast
}}}+\left( h_{1l^{\ast }}^{\ast },\psi _{l^{\ast }}\right) _{\Gamma
_{l^{\ast }}}.%
\end{array}%
\end{equation*}%
Invoking the relations in (\ref{10}) and the variational equation (\ref{8}%
), which is satisfied by $\left[ u_{0},h_{11},\ldots ,h_{1K},w_{1}\right] $\\
(where again vector $\left[ \varphi ,\mathbf{\psi },\xi \right] $ is given
by (\ref{vec})), we have the relation%
\begin{equation*}
\left\langle \frac{\partial h_{0j^{\ast }}}{\partial n_{j^{\ast }}}%
,g\right\rangle _{\partial \Gamma _{j^{\ast }}\cap \partial \Gamma _{l^{\ast
}}}=-\left\langle \frac{\partial h_{0l^{\ast }}}{\partial n_{l^{\ast }}}%
,g\right\rangle _{\partial \Gamma _{j^{\ast }}\cap \partial \Gamma _{l^{\ast
}}}\text{, \ for all }g\in H_{0}^{\frac{1}{2}+\epsilon }(\partial \Gamma
_{j^{\ast }}\cap \partial \Gamma _{l^{\ast }})\text{. }
\end{equation*}%
Since $H_{0}^{\frac{1}{2}+\epsilon }(\partial \Gamma _{j^{\ast }}\cap
\partial \Gamma _{l^{\ast }})$ is dense in $H^{\frac{1}{2}}(\partial \Gamma _{j^{\ast }}\cap \partial \Gamma _{l^{\ast }})$, we deduce
now that 
\begin{equation}
\frac{\partial h_{0j^{\ast }}}{\partial n_{j^{\ast }}}=-\frac{\partial
h_{0l^{\ast }}}{\partial n_{l^{\ast }}}\text{, \ for }\partial \Gamma
_{j^{\ast }}\cap \partial \Gamma _{l^{\ast }}\neq \emptyset .  \label{last}
\end{equation}
Collecting (\ref{9})-(\ref{14}) and (\ref{17}), (\ref{19}) and (\ref{last}),
we have that the obtained variable%
\begin{equation*}
\left[ u_{0},h_{01},h_{11},\ldots ,h_{0K},h_{1K},w_{0},w_{1}\right] \in D(%
\mathbf{A}),
\end{equation*}%
and solves the resolvent equation (\ref{a1}). This concludes the proof of
Theorem \ref{well}, upon application of the Lumer-Phillips Theorem.

\section{Strong Stability-Proof of Theorem \ref{SS} }
In this section, our main aim is to address the issue of asymptotic behavior of the solution that we stated in Section 2. In this regard, we show that the system given in (\ref{2a})-(\ref{IC}) is strongly stable. Our proof will be independent of the compactness or noncompactness of the resolvent of $\mathbf{A}$ (see Remark \ref {remark_1}.)  It will hinge on an ultimate appeal to the following well known result :

\begin{theorem} 
\label{AB} (\cite{A-B}) Let ${T(t)}_{t\geq 0}$ be a bounded $C_{0}$-semigroup on a reflexive Banach space X, with generator $\mathbf{A}$. Assume that $\sigma_p(\mathbf{A})\cap i\mathbb{R}=\emptyset$, where $\sigma_p(\mathbf{A})$ is the point spectrum of $\mathbf{A}$. If $\sigma(\mathbf{A})\cap i\mathbb{R}$ is countable then ${T(t)}_{t\geq 0}$ is strongly stable.
\end{theorem}
The proof of this theorem entails the elimination of all three parts of the spectrum of the generator $\mathbf{A}$ from the imaginary axis. For this, we will give the necessary analysis on the spectrum in the following subsection.

\subsection{Spectral Analysis on the generator $\mathbf{A}$}
Since we wish to satisfy the conditions of Theorem \ref{AB}, we will prove that $\sigma(\mathbf{A})\cap i\mathbb{R}=\emptyset$ which is equivalent to show that $$ i\mathbb{R}\subset \rho (\mathbf{A}).$$ To do this, we start with the following Proposition:
\begin{proposition}
\label{invert} With generator $\mathbf{A}:D(\mathbf{A})\subset \mathbf{H}%
\rightarrow \mathbf{H}$ given in (\ref{4a})-(\ref{dom}), the point $0\in
\rho (\mathbf{A).}$ That is, $\mathbf{A}$ is boundedly invertible.
\end{proposition}

\begin{proof}
Given $\Phi ^{\ast }=\left[ u_{0}^{\ast },h_{01}^{\ast },h_{11}^{\ast
},\ldots ,h_{0K}^{\ast },h_{1K}^{\ast },w_{0}^{\ast },w_{1}^{\ast }\right]
\in \mathbf{H,}$ we take up the task of finding $\Phi =\left[
u_{0},h_{01},h_{11},\ldots ,h_{0K},h_{1K},w_{0},w_{1}\right] \in D(\mathbf{A}%
)$ which solves%
\begin{equation}
\mathbf{A}\Phi =\Phi ^{\ast },  \label{43}
\end{equation}%
or%
\begin{equation}
\left[ 
\begin{array}{c}
\Delta u_{0} \\ 
h_{11} \\ 
-\frac{\partial u_{0}}{\partial \nu }|_{\Gamma _{1}}+(\Delta -I)h_{01}+\frac{%
\partial w_{0}}{\partial \nu }|_{\Gamma _{1}} \\ 
\vdots  \\ 
h_{1K} \\ 
-\frac{\partial u_{0}}{\partial \nu }|_{\Gamma _{K}}+(\Delta -I)h_{0K}+\frac{%
\partial w_{0}}{\partial \nu }|_{\Gamma _{K}} \\ 
w_{1} \\ 
\Delta w_{0}%
\end{array}%
\right] =\left[ 
\begin{array}{c}
u_{0}^{\ast } \\ 
h_{01}^{\ast } \\ 
h_{11}^{\ast } \\ 
\vdots  \\ 
h_{0K}^{\ast } \\ 
h_{1K}^{\ast } \\ 
w_{0}^{\ast } \\ 
w_{1}^{\ast }%
\end{array}%
\right] .  \label{44}
\end{equation}%
From the thin and thick wave component of this equation we see that%
\begin{equation}
w_{1}=w_{0}^{\ast }\in H^{1}(\Omega_s )  \label{45}
\end{equation}%
\begin{equation}
h_{1j}=h_{0j}^{\ast }\in H^{1}(\Gamma _{j}),\text{ \ \ \ for \  }1\leq j\leq
K \label{46}
\end{equation}%
Moreover, from the heat and thick wave components of (\ref{44}), and the
domain criterion (A.iii), we have that the solution component $u_{0}$ should
satisfy the following BVP:%
\begin{equation}
\left\{ 
\begin{array}{c}
\Delta u_{0}=u_{0}^{\ast }\text{ \ \ \ \ in \ }\Omega _{f} \\ 
u_{0}|_{\Gamma _{f}}=0 \\ 
u_{0}|_{\Gamma s}=w_{0}^{\ast }|_{\Gamma _{s}}%
\end{array}%
\right.   \label{47}
\end{equation}%
Solving this BVP, and estimating its solution, in part by the Sobolev Trace
Theorem, we have%
\begin{equation}
\left\Vert u_{0}\right\Vert _{H_{\Gamma _{f}}^{1}(\Omega _{f})}+\left\Vert
\Delta u_{0}\right\Vert _{\Omega _{f}}\leq C\left[ \left\Vert u_{0}^{\ast
}\right\Vert _{\Omega _{f}}+\left\Vert w_{0}^{\ast }\right\Vert
_{H^{1}(\Omega _{s})}\right] .  \label{48}
\end{equation}%
In turn, the use of this estimate in an integration by parts gives%
\begin{equation}
\left\Vert \frac{\partial u_{0}}{\partial \nu }\right\Vert _{H^{-\frac{1}{2}%
}(\partial \Omega _{f})}\leq C\left[ \left\Vert u_{0}^{\ast }\right\Vert
_{\Omega _{f}}+\left\Vert w_{0}^{\ast }\right\Vert _{H^{1}(\Omega _{s})}%
\right].   \label{49}
\end{equation}%
In addition, with the space $\mathcal{V}$ as in (\ref{V}), we set%
\begin{equation}
\chi \equiv \left\{ \left[ \psi ,\xi \right] \in \mathcal{V}\times
H^{1}(\Omega _{s}):\psi _{j}=\xi |_{\Gamma _{j}}\text{ \ \ for \ }1\leq
j\leq K\text{\ }\right\} .  \label{50}
\end{equation}%
With this space in hand, and with the thin-wave and thick-wave components of
equation (\ref{44}) in mind, we consider the variational relation%
\begin{eqnarray}
&&(\nabla w_{0},\nabla \xi )_{\Omega _{s}}  \notag \\
&&+\sum\limits_{j=1}^{K}\left[ (\nabla h_{0j},\nabla \psi _{j})_{\Gamma
_{j}}+(h_{0j},\psi _{j})_{\Gamma _{j}}\right]   \notag \\
&=&-(w_{1}^{\ast },\xi )_{\Omega _{s}}  \notag \\
&&-\sum\limits_{j=1}^{K}\left[ (h_{1j}^{\ast },\psi _{j})_{\Gamma _{j}}+(%
\frac{\partial u_{0}}{\partial \nu },\psi _{j})_{\Gamma _{j}}\right], 
\label{51}
\end{eqnarray}%
for every $\left[ \psi ,\xi \right] \in \chi $ where the term $\frac{\partial u_{0}}{\partial \nu }|_{\Gamma _{s}}$ is from (\ref{49}). Since the bilinear form $%
b(\cdot ,\cdot ):\chi \rightarrow 
\mathbb{R}
,$ given by%
\begin{equation}
b(\left[ \psi ,\xi \right] ,\left[ \widetilde{\psi },\widetilde{\xi }\right]
)=(\nabla \xi ,\nabla \widetilde{\xi })_{\Omega _{s}}+\sum\limits_{j=1}^{K}%
\left[ (\nabla \psi _{j},\nabla \widetilde{\psi }_{j})_{\Gamma _{j}}+(\psi
_{j},\widetilde{\psi }_{j})_{\Gamma _{j}}\right]   \label{52}
\end{equation}%
for every $\left[ \psi ,\xi \right] ,\left[ \widetilde{\psi },\widetilde{\xi 
}\right] \in \chi ,$ is continuous and $\chi $-elliptic, then by
Lax-Milgram, there exists a unique solution 
\begin{equation}
\phi =\left[ (h_{01},h_{02},\ldots ,h_{0K}),w_{0}\right] \in \chi 
\label{53}
\end{equation}%
to the variational relation (\ref{51}). To show that the obtained $\left[
u_{0},[h_{01},h_{11},\ldots ,h_{0K},h_{1K}],w_{0},w_{1}\right] \in \mathbf{H}
$ is in $D(\mathbf{A})$ and satisfies the equation (\ref{44}):\newline
Proceeding very much as we did in the proof of Theorem \ref{well}, we take
in (\ref{51})%
\begin{equation*}
\left[ \psi ,\xi \right] =\left[ \left[ 0,0,...,0\right] ,\varphi \right] ,
\end{equation*}%
where $\varphi \in \mathcal{D}(\Omega _{s}).$ This gives%
\begin{equation*}
(\nabla w_{0},\nabla \xi )_{\Omega _{s}}=-(w_{1}^{\ast },\xi )_{\Omega _{s}},
\end{equation*}%
whence we obtain%
\begin{equation}
-\Delta w_{0}=-w_{1}^{\ast }\text{ \ \ \ in \ \ }\Omega _{s},  \label{54}
\end{equation}%
with%
\begin{eqnarray}
\left\Vert \Delta w_{0}\right\Vert _{\Omega _{s}}+\left\Vert \frac{\partial
w_{0}}{\partial \nu }\right\Vert _{H^{-\frac{1}{2}}(\Gamma _{s})} &\leq &C%
\left[ \left\Vert w_{1}^{\ast }\right\Vert _{\Omega _{s}}+\left\Vert
w_{0}\right\Vert _{H^{1}(\Omega _{s})}\right]   \notag \\
&\leq &C\left\Vert \left[ u_{0}^{\ast },[h_{01}^{\ast },h_{11}^{\ast
},\ldots ,h_{0K}^{\ast },h_{1K}^{\ast }],w_{0}^{\ast },w_{1}^{\ast }\right]
\right\Vert _{\mathbf{H}},  \label{55}
\end{eqnarray}%
after using (\ref{53}). In turn, using aforesaid right continuous inverse $%
\gamma _{s}^{+}\in \mathcal{L}(H^{\frac{1}{2}}(\Gamma _{s}),H^{1}(\Omega
_{s})),$ let in (\ref{51}), test function%
\begin{equation*}
\left[ \psi ,\xi \right] =\left[ \left[ (\psi _{1})_{ext},...,(\psi
_{K})_{ext}\right] ,\gamma _{s}^{+}\left( \sum\limits_{j=1}^{K}\left( \psi
_{j}\right) _{ext}\right) \right] \in \chi ,
\end{equation*}%
where each $\psi _{j}\in H_{0}^{1}(\Gamma _{j})$ $(1\leq j\leq K),$ and each 
$\left( \psi _{j}\right) _{ext}$ is as in (\ref{15}). Applying this function
to (\ref{51}), integrating by parts and invoking (\ref{54}), we have%
\begin{eqnarray*}
&&-(\Delta w_{0},\xi )_{\Omega _{s}}-\left\langle \frac{\partial w_{0}}{%
\partial \nu },\xi |_{\Gamma _{s}}\right\rangle _{\Gamma _{s}} \\
&&+\sum\limits_{j=1}^{K}\left[ (\nabla h_{0j},\nabla \psi _{j})_{\Gamma
_{j}}+(h_{0j},\psi _{j})_{\Gamma _{j}}\right]  \\
&=&-\sum\limits_{j=1}^{K}\left[ \left\langle \frac{\partial u_{0}}{\partial
\nu },\psi _{j}\right\rangle _{\Gamma _{j}}+(h_{1j}^{\ast },\psi
_{j})_{\Gamma _{j}}\right] -(w_{1}^{\ast },\xi )_{\Omega _{s}}.
\end{eqnarray*}%
Again, as each $\psi _{j}\in H_{0}^{1}(\Gamma _{j})$ is arbitrary, we deduce
that each $h_{0j}$ solves the thin-wave equation%
\begin{equation}
-\Delta h_{0j}+h_{0j}-\frac{\partial w_{0}}{\partial \nu }+\frac{\partial
u_{0}}{\partial \nu }=-h_{1j}^{\ast },\text{\  in \ \ } \Gamma _{j},\text{ \ } 1\leq j\leq K.  \label{56}
\end{equation}%
A subsequent integration by parts, and invocation of (\ref{49}), (\ref{53}) and (\ref{55}%
), give for $1\leq j\leq K,$%
\begin{equation*}
\left\Vert \Delta h_{0j}\right\Vert _{\Gamma _{j}}+\left\Vert \frac{\partial
h_{0j}}{\partial n_{j}}\right\Vert _{H^{-\frac{1}{2}}(\partial \Gamma _{j})}
\end{equation*}%
\begin{equation}
\leq C\left\Vert \left[ u_{0}^{\ast },[h_{01}^{\ast },h_{11}^{\ast },\ldots
,h_{0K}^{\ast },h_{1K}^{\ast }],w_{0}^{\ast },w_{1}^{\ast }\right]
\right\Vert _{\mathbf{H}}.  \label{57}
\end{equation}%
Now, proceeding as in the final stage of the proof of Theorem \ref{well}:
let fixed indices  $j^{\ast },l^{\ast }$, $1\leq j^{\ast },l^{\ast }\leq K$,
satisfy $\partial \Gamma _{j^{\ast }}\cap $ $\partial \Gamma _{l^{\ast
}}\neq \emptyset $. Given function $g$ $\in $ $H_{0}^{\frac{1}{2}+\epsilon
}(\partial \Gamma _{j^{\ast }}\cap $ $\partial \Gamma _{l^{\ast }}),$ we
invoke the associated functions $\psi _{j^{\ast }}\in H^{1+\epsilon }(\Gamma
_{j^{\ast }})$ and $\psi _{l^{\ast }}\in H^{1+\epsilon }(\Gamma _{l^{\ast }})
$ as in (\ref{g}), also $\Upsilon \in H^{1}(\Gamma _{s})$\ as in (\ref{sig}%
). With these functions, and said continuous right inverse $\gamma
_{s}^{+}\in \mathcal{L}(H^{\frac{1}{2}}(\Gamma _{s}),H^{1}(\Omega _{s}))$, we
consider test function%
\begin{equation*}
\left[ \psi ,\xi \right] =\left[ \left[ 0,...,\psi _{j^{\ast }},0,...0,\psi
_{l^{\ast }},...,0\right] ,\gamma _{s}^{+}\left( \Upsilon \right) \right]
\in \chi. 
\end{equation*}%
Applying this test function to the variational relation (\ref{51}), and
subsequently invoking (\ref{54}), we obtain%
\begin{eqnarray*}
&&-\left\langle \frac{\partial w_{0}}{\partial \nu },\xi |_{\Gamma
_{s}}\right\rangle _{\Gamma _{s}}+(\nabla h_{0j^{\ast }},\nabla \psi
_{j^{\ast }})_{\Gamma _{j}^{\ast }}+(h_{0j^{\ast }},\psi _{j^{\ast
}})_{\Gamma _{j^{\ast }}} \\
&&+(\nabla h_{0l^{\ast }},\nabla \psi _{l^{\ast }})_{\Gamma _{l}^{\ast
}}+(h_{0l^{\ast }},\psi _{l^{\ast }})_{\Gamma _{l^{\ast }}} \\
&=&-(h_{1j^{\ast }},\psi _{j^{\ast }})_{\Gamma _{j^{\ast }}}-\left\langle 
\frac{\partial u_{0}}{\partial \nu },\psi _{j^{\ast }}\right\rangle _{\Gamma
_{j}^{\ast }} \\
&&-(h_{1l}^{\ast },\psi _{l^{\ast }})_{\Gamma _{l}^{\ast }}-\left\langle 
\frac{\partial u_{0}}{\partial \nu },\psi _{l^{\ast }}\right\rangle _{\Gamma
_{l}^{\ast }}.
\end{eqnarray*}%
Integrating by parts with respect to the thin wave components, and invoking (%
\ref{56}) and (\ref{g}), we then have%
\begin{equation*}
\left\langle \frac{\partial h_{0j^{\ast }}}{\partial n_{j^{\ast }}}%
,g\right\rangle _{\partial \Gamma _{j^{\ast }}\cap \partial \Gamma _{l^{\ast
}}}+\left\langle \frac{\partial h_{0l^{\ast }}}{\partial n_{l^{\ast }}}%
,g\right\rangle _{\partial \Gamma _{j^{\ast }}\cap \partial \Gamma _{l^{\ast
}}}=0\text{.}
\end{equation*}%
Since $g\in H_{0}^{\frac{1}{2}+\epsilon }(\partial \Gamma _{j^{\ast }}\cap
\partial \Gamma _{l^{\ast }})$ is arbitrary, a density argument yields%
\begin{equation}
\left\langle \frac{\partial h_{0j^{\ast }}}{\partial n_{j^{\ast }}}%
,g\right\rangle _{\partial \Gamma _{j^{\ast }}\cap \partial \Gamma _{l^{\ast
}}}=-\left\langle \frac{\partial h_{0l^{\ast }}}{\partial n_{l^{\ast }}}%
,g\right\rangle _{\partial \Gamma _{j^{\ast }}\cap \partial \Gamma _{l^{\ast
}}}\text{, \ }\forall \text{ }j^{\ast },l^{\ast }\text{ \ \ },1\leq j^{\ast
},l^{\ast }\leq K  \label{58}
\end{equation}%
such that $\partial \Gamma _{j^{\ast }}\cap \partial \Gamma _{l^{\ast }}\neq
\emptyset .$ Collecting (\ref{45}), (\ref{46}), (\ref{48}), (\ref{49}), (\ref{53}), (\ref{54}), (\ref%
{56})-(\ref{58}), we have now that the obtained $\left[ u_{0},[h_{01},h_{11},\ldots
,h_{0K},h_{1K}],w_{0},w_{1}\right]\in D(\mathbf{A})$ satisfies
the equation (\ref{43}) for arbitrary $\Phi ^{\ast }\in \mathbf{H.}$ Since also $%
\mathbf{A}:D(\mathbf{A})\subset \mathbf{H}\rightarrow \mathbf{H}$ is
dissipative (and so injective), we conclude that $\mathbf{A}$ is boundedly invertible.
\end{proof}

In what follows, we will need the Hilbert space adjoint of $\mathbf{A}:D(%
\mathbf{A})\subset \mathbf{H}\rightarrow \mathbf{H}$ which can be readily
computed:

\begin{proposition}
\label{adj} The Hilbert space adjoint \ $\mathbf{A}^{\ast }:D(\mathbf{A}%
^{\ast })\subset \mathbf{H}\rightarrow \mathbf{H}$ of the thick wave-thin
wave-heat generator is given as,%
\begin{equation*}
\mathbf{A}^{\ast }=\left[ 
\begin{array}{cccccccc}
\Delta  & 0 & 0 & 0 & 0 & 0 & 0 & 0 \\ 
0 & 0 & -I & \cdots  & 0 & 0 & 0 & 0 \\ 
-\frac{\partial }{\partial \nu }|_{\Gamma _{1}} & (I-\Delta ) & 0 & \cdots 
& 0 & 0 & -\frac{\partial }{\partial \nu }|_{\Gamma _{1}} & 0 \\ 
\vdots  & \vdots  & \vdots  & \cdots  & \vdots  & \vdots  & \vdots  & \vdots 
\\ 
0 & 0 & 0 & \cdots  & 0 & -I & 0 & 0 \\ 
-\frac{\partial }{\partial \nu }|_{\Gamma _{K}} & 0 & 0 & \cdots  & (I-\Delta ) & 0 & -\frac{\partial }{\partial \nu }|_{\Gamma _{1}}& 0 \\ 
0 & 0 & 0 & \cdots  & 0 & 0 & 0 & -I \\ 
0 & 0 & 0 & \cdots  & 0 & 0 & -\Delta  & 0%
\end{array}%
\right] ;
\end{equation*}
\end{proposition}

\bigskip where%
\begin{equation*}
\begin{array}{l}
D(\mathbf{A}^{\ast })=\left\{ \left[ u_{0},h_{01},h_{11},\ldots
,h_{0K},h_{1K},w_{0},w_{1}\right] \in \mathbf{H}:\right.  \\ 
\text{ \ \ (A}^{\ast }\text{.i) }u_{0}\in H^{1}(\Omega _{f})\text{, }%
h_{1j}\in H^{1}(\Gamma _{j})\text{ for }1\leq j\leq K\text{, }w_{1}\in
H^{1}(\Omega _{s})\text{;} \\ 
\text{ \ }\left. \text{(A}^{\ast }\text{.ii) (a) }\Delta u_{0}\in
L^{2}(\Omega _{f})\text{, }\Delta w_{0}\in L^{2}(\Omega _{s})\text{, (b) }%
-\Delta h_{0j}-\frac{\partial u_{0}}{\partial \nu }|_{\Gamma _{j}}-\frac{%
\partial w_{0}}{\partial \nu }|_{\Gamma _{j}}\in L^{2}(\Gamma _{j})\text{ \
for\ } 1\leq j\leq K\text{;}\right.  \\ 
\text{ \ \ \ \ \ \ (c) }\left. \dfrac{\partial h_{0j}}{\partial n_{j}}%
\right\vert _{\partial \Gamma _{j}}\in H^{-\frac{1}{2}}(\partial \Gamma _{j})%
\text{, \ for\ } 1\leq j\leq K\text{;} \\ 
\text{ \ }\left. \text{(A}^{\ast }\text{.iii) }u_{0}|_{\Gamma _{f}}=0,\ \
u_{0}|_{\Gamma _{j}}=h_{1j}=w_{1}|_{\Gamma _{j}},\ \text{for }1\leq j\leq K%
\text{;}\right.  \\ 
\text{ \ }\left. \text{(A}^{\ast }\text{.iv) For }1\leq j\leq K\text{: }%
\right.  \\ 
\text{ \ \ \ \ \ \ (a) }h_{1j}|_{\partial \Gamma _{j}\cap \partial \Gamma
_{l}}=h_{1l}|_{\partial \Gamma _{j}\cap \partial \Gamma _{l}}\text{ on \ }%
\partial \Gamma _{j}\cap \partial \Gamma _{l}\text{, for all }1\leq l\leq K%
\text{ such that }\partial \Gamma _{j}\cap \partial \Gamma _{l}\neq
\emptyset ; \\ 
\text{ \ \ \ \ \ \ \ }\left. \text{(b) }\left. \dfrac{\partial h_{0j}}{%
\partial n_{j}}\right\vert _{\partial \Gamma _{j}\cap \partial \Gamma
_{l}}=-\left. \dfrac{\partial h_{0_{l}}}{\partial n_{l}}\right\vert
_{\partial \Gamma _{j}\cap \partial \Gamma _{l}}\text{ on \ }\partial \Gamma
_{j}\cap \partial \Gamma _{l}\text{, for all }1\leq l\leq K\text{ such that }%
\partial \Gamma _{j}\cap \partial \Gamma _{l}\neq \emptyset \right\} .%
\end{array}%
\end{equation*}%
Now, we continue with analyzing the point and continuous spectra of the generator $\mathbf{A}$: 

\begin{lemma}
\label{point-cont-spec} The point $\sigma_p(\mathbf{A})$ and continuous spectra $\sigma_c(\mathbf{A})$ of $\mathbf{A}$
have empty intersection with $i%
\mathbb{R}
$.
\end{lemma}

\textbf{Proof. } To prove this, it will be
enough to show that $i%
\mathbb{R}
\backslash \{0\}$ has empty intersection with the approximate spectrum of $\mathbf{A}$;
see e.g., Theorem 2.27, pg. 128 of \cite{F}. To this end, given $\beta \neq 0,$
suppose that $i\beta $ is in the approximate spectrum of $\mathbf{A.}$ Then
there exist sequences%
\begin{equation}
\{\Phi _{n}\}=\left\{ \left[ 
\begin{array}{c}
u_{n} \\ 
h_{1n} \\ 
\xi _{1n} \\ 
\vdots  \\ 
h_{Kn} \\ 
\xi _{Kn} \\ 
w_{0n} \\ 
w_{1n}%
\end{array}%
\right] \right\} \subseteq D(\mathbf{A});\text{ \ \ \ \ \ \ \ }\{(i\beta I-%
\mathbf{A)}\Phi _{n}\}=\left\{ \left[ 
\begin{array}{c}
u_{n}^{\ast } \\ 
\varphi _{1n}^{\ast } \\ 
\psi _{1n}^{\ast } \\ 
\vdots  \\ 
\varphi _{Kn}^{\ast } \\ 
\psi _{Kn}^{\ast } \\ 
w_{0n}^{\ast } \\ 
w_{1n}^{\ast }%
\end{array}%
\right] \right\} \subseteq \mathbf{H}\text{\ ,}  \label{60}
\end{equation}%
which satisfy for $n=1,2,...,$%
\begin{equation}
\left\Vert \Phi _{n}\right\Vert _{\mathbf{H}}=1,\text{ \ \ \ }\left\Vert
(i\beta I-\mathbf{A)}\Phi _{n}\right\Vert _{\mathbf{H}}<\frac{1}{n}.
\label{61}
\end{equation}%
As such, each $\Phi _{n}$ solves the following static system:%
\begin{equation}
\left\{ 
\begin{array}{c}
i\beta u_{n}-\Delta u_{n}=u_{n}^{\ast }\text{ \ \ in \ }\Omega _{f} \\ 
u_{n}|_{\Gamma _{f}}=0\text{ \ \ on \ }\Gamma _{f}%
\end{array}%
\right.   \label{62}
\end{equation}%
For $1\leq j\leq K,$%
\begin{equation}
\left\{ 
\begin{array}{c}
i\beta h_{jn}-\xi _{jn}=\varphi _{jn}^{\ast }\text{ \ \ in \ }\Gamma _{j} \\ 
-\beta ^{2}h_{jn}-\Delta h_{jn}+h_{jn}+\frac{\partial u_{n}}{\partial \nu }-%
\frac{\partial w_{0n}}{\partial \nu }=\psi _{jn}^{\ast }+i\beta \varphi
_{jn}^{\ast }\text{ \ \ in \ }\Gamma _{j}%
\end{array}%
\right.   \label{63}
\end{equation}%
Also%
\begin{equation}
\left\{ 
\begin{array}{c}
i\beta w_{0n}-w_{1n}=w_{0n}^{\ast }\text{ \ \ in \ \ }\Omega _{s} \\ 
-\beta ^{2}w_{0n}-\Delta w_{0n}=w_{1n}^{\ast }+i\beta w_{0n}^{\ast }\text{ \
\ in \ \ }\Omega _{s}%
\end{array}%
\right.   \label{64}
\end{equation}%
and again for $1\leq j\leq K,$%
\begin{equation}
\left\{ 
\begin{array}{c}
u_{n}|_{\Gamma _{j}}=\xi _{jn}=w_{1n}|_{\Gamma _{j}} \\ 
\left. \dfrac{\partial h_{nj}}{\partial n_{j}}\right\vert _{\partial \Gamma
_{j}\cap \partial \Gamma _{l}}=-\left. \dfrac{\partial h_{nl}}{\partial n_{l}%
}\right\vert _{\partial \Gamma _{j}\cap \partial \Gamma _{l}}\text{  for all 
}1\leq l\leq K\text{ such that }\partial \Gamma _{j}\cap \partial \Gamma
_{l}\neq \emptyset .
\end{array}%
\right.   \label{65}
\end{equation}%
Now the left part of the proof of Lemma \ref{point-cont-spec} will be given in five steps:\newline\\
\underline{\textbf{STEP 1:}}\textbf{\ (Estimating the heat component of }$%
\Phi _{n}$\textbf{) }\\ \\Proceeding as we did in establishing the dissipativity
of $\mathbf{A}:D(\mathbf{A})\subset \mathbf{H}\rightarrow \mathbf{H}$, (see
relations (\ref{ten}) and (\ref{dissi})), if we denote 
\begin{equation*}
\Phi _{n}^{\ast }=(i\beta I-\mathbf{A)}\Phi _{n}
\end{equation*}%
then from the relation%
\begin{equation*}
\left( (i\beta I-\mathbf{A)}\Phi _{n},\Phi _{n}\right) _{\mathbf{H}}=(\Phi
_{n}^{\ast },\Phi _{n})_{\mathbf{H}},
\end{equation*}%
we obtain%
\begin{equation}
\left\Vert \nabla u_{n}\right\Vert _{\Omega _{f}}^{2}=\text{Re}(\Phi
_{n}^{\ast },\Phi _{n})_{\mathbf{H}}.  \label{66}
\end{equation}
From (\ref{61}), we then have%
\begin{equation}
\underset{n\rightarrow \infty }{\lim }u_{n}=0\text{ \ \ in \ \ }H^{1}(\Omega
_{f}).  \label{67}
\end{equation}%
In turn, via the thin wave resolvent condition in (\ref{63}) and boundary
conditions in (\ref{65}), we have for $1\leq j\leq K$ 
\begin{equation*}
h_{jn}=-\frac{i}{\beta }u_{n}|_{\Gamma _{j}}-\frac{i}{\beta }\varphi
_{jn}^{\ast }\text{ \ \ in \ }\Gamma _{j}.
\end{equation*}%
From this relation, we can then invoke (\ref{67}), the Sobolev Trace Map,
and (\ref{61}), to have%
\begin{equation}
\underset{n\rightarrow \infty }{\lim }h_{jn}=0\text{ \ \ in \ \ }H^{\frac{1}{%
2}}(\Gamma _{j})  \label{68}
\end{equation}%
for $1\leq j\leq K.$ Moreover, an integration by parts, with respect to the
heat equation (\ref{62}), gives the estimate%
\begin{eqnarray*}
\left\Vert \frac{\partial u_{n}}{\partial \nu }\right\Vert _{H^{-\frac{1}{2}%
}(\partial \Omega _{f})} &\leq &C\left[ \left\Vert \nabla u_{n}\right\Vert
_{\Omega _{f}}+\left\Vert \Delta u_{n}\right\Vert _{\Omega _{f}}\right]  \\
&\leq &C\left[ \left\Vert \nabla u_{n}\right\Vert _{\Omega _{f}}+\left\Vert
i\beta u_{n}-u_{n}^{\ast }\right\Vert _{\Omega _{f}}\right] .
\end{eqnarray*}%
Now, invoking (\ref{66}) and (\ref{61}) gives%
\begin{equation}
\underset{n\rightarrow \infty }{\lim }\frac{\partial u_{n}}{\partial \nu }=0%
\text{ \ \ in \ \ }H^{-\frac{1}{2}}(\Gamma _{j}) . \label{69}
\end{equation}%
\newline
\underline{\textbf{STEP 2:}} We start here by defining the "Dirichlet" map $%
D_{s}:L^{2}(\Gamma _{s})\rightarrow L^{2}(\Omega _{s})$\ via%
\begin{equation*}
D_{s}g=f\Longleftrightarrow \left\{ 
\begin{array}{c}
\Delta f=0\text{ \ in \ }\Omega _{s} \\ 
f|_{\Gamma _{s}}=g\text{ \ on \ }\Gamma _{s}.%
\end{array}%
\right. 
\end{equation*}%
We know by the Lax-Milgram Theorem%
\begin{equation}
D_{s}\in \mathcal{L}(H^{\frac{1}{2}}(\Gamma _{s}),H^{1}(\Omega _{s})).
\label{70}
\end{equation}%
Therewith, considering the resolvent relations in (\ref{64}), we set
\begin{equation}
z_{n}\equiv w_{0n}+\frac{i}{\beta}D_{s}[u_{n}|_{\Gamma _{s}}+w_{0n}^{\ast }|_{\Gamma _{s}}],  \label{71}
\end{equation}%
and so from (\ref{64}) $z_{n}$ satisfies the following BVP:%
\begin{equation}
\left\{ 
\begin{array}{c}
-\beta ^{2}z_{n}-\Delta z_{n}=w_{1n}^{\ast }+i\beta w_{0n}^{\ast }-i\beta D_{s}[u_{n}|_{\Gamma _{s}}+w_{0n}^{\ast }|_{\Gamma _{s}}]
\text{ \ \ in \ \ }\Omega _{s} \\ 
z_{n}|_{\Gamma _{s}}=0\text{ \ \ on \ \ }\Gamma _{s}.%
\end{array}%
\right. \label{bvp}
\end{equation}%
Since $\Omega _{s}$ is convex, then $z_{n}\in H^{2}(\Omega _{s}).$ See e.g., Theorem 3.2.1.2, pg. 147 of \cite{GV}. In
consequence, we can apply the static version of the well-known wave identity
which is often used in PDE control theory-- [see (Proposition 7 (ii) of \cite{AG1}), 
\cite{chen}, \cite{trigg}. To wit, let $m(x)$
be any $[C^{2}(\overline{\Omega _{s}})]^{3}$- vector field with associated
Jacobian matrix%
\begin{equation*}
\left[ M(x)\right] _{ij}=\frac{\partial m_{i}(x)}{\partial x_{j}},\text{ \ \ 
}1\leq i,j\leq 3
\end{equation*}%
Therewith, we have%
\begin{eqnarray}
&&\int\limits_{\Omega _{s}}M\nabla z_{n}\cdot \nabla z_{n}d\Omega _{s} 
\notag \\
&=&-\text{Re}\int\limits_{\Gamma _{s}}\frac{\partial z_{n}}{\partial \nu }%
m\cdot \nabla \overline{z_{n}}d\Gamma _{s}  \notag \\
&&-\frac{\beta ^{2}}{2}\int\limits_{\Gamma _{s}}\left\vert z_{n}\right\vert
^{2}m\cdot \nu d\Gamma _{s}+\frac{1}{2}\int\limits_{\Gamma _{s}}\left\vert
\nabla z_{n}\right\vert ^{2}m\cdot \nu d\Gamma _{s}  \notag \\
&&+\frac{1}{2}\int\limits_{\Omega _{s}}\{\left\vert \nabla z_{n}\right\vert
^{2}-\beta ^{2}\left\vert z_{n}\right\vert ^{2}\}\text{div}(m)d\Omega _{s} 
\notag \\
&&+\text{Re}\int\limits_{\Omega _{s}}\left[ F_{\beta }^{\ast }-i\beta
D_{s}[u_{n}|_{\Gamma _{s}}+w_{0n}^{\ast }|_{\Gamma _{s}}]\right] m\cdot \nabla \overline{z_{n}}d\Omega
_{s},  \label{72}
\end{eqnarray}%
where 
\begin{equation}
F_{\beta }^{\ast }=(\text{Re}w_{1n}^{\ast }-\beta I_{m}w_{0n}^{\ast
})+i(I_{m}w_{1n}^{\ast }+\beta \text{Re}w_{0n}^{\ast }).  \label{73}
\end{equation}%
Again, relation (\ref{72}) holds for any $C^{2}-$vector field $m(x).$ We now
specify it to be the smooth vector field of Lemma 1.5.1.9, pg. 40 of
\cite{GV}. Namely, for some $\delta >0,$ the $C^{\infty }$ vector
field $m(x)$ satisfies%
\begin{equation}
-m(x)\cdot \nu \geq \delta \text{ \ \ a.e. \ on  }\Gamma _{s}  \label{74}
\end{equation}%
Specifying this vector field in (\ref{72}), and considering that $%
z_{n}|_{\Gamma _{s}}=0,$ we have then%
\begin{eqnarray}
&&-\frac{1}{2}\int\limits_{\Gamma _{s}}\left\vert \frac{\partial z_{n}}{%
\partial \nu }\right\vert ^{2}m\cdot \nu d\Gamma _{s}  \notag \\
&=&\int\limits_{\Omega _{s}}M\nabla z_{n}\cdot \nabla z_{n}d\Omega _{s} 
\notag \\
&&+\frac{1}{2}\int\limits_{\Omega _{s}}\{\beta ^{2}\left\vert
z_{n}\right\vert ^{2}-\left\vert \nabla z_{n}\right\vert ^{2}\}d\Omega _{s} 
\notag \\
&&-\text{Re}\int\limits_{\Omega _{s}}\left[ F_{\beta }^{\ast }-i\beta
D_{s}[u_{n}|_{\Gamma _{s}}+w_{0n}^{\ast }|_{\Gamma _{s}}]\right] m\cdot \nabla \overline{z_{n}}d\Omega
_{s}.  \label{75}
\end{eqnarray}%
Estimating this relation via (\ref{61}), ((\ref{67}), \ref{71}), (\ref{70}) and the
Sobolev Trace map, we then have%
\begin{equation}
\int\limits_{\Gamma _{s}}\left\vert \frac{\partial z_{n}}{\partial \nu }%
\right\vert ^{2}d\Gamma _{s}\leq C_{\delta ,\beta ,m},  \label{76}
\end{equation}%
where positive constant $C_{\delta ,\beta ,m}$ is independent of $n=1,2,...$%
\newline\\
\underline{\textbf{STEP 3:}} \textbf{( An energy estimate for $h_{jn}$ )} \\ \\We multiply
both sides of the thin wave $h_{jn}-$ equation (\ref{63}) by $h_{jn},$
integrate and subsequently integrate by parts to have for $1\leq j\leq K,$%
\begin{eqnarray}
\int\limits_{\Gamma _{j}}\left\vert \nabla h_{jn}\right\vert ^{2}d\Gamma
_{j} &=&\int\limits_{\Gamma _{j}}\frac{\partial w_{0n}}{\partial \nu }%
h_{jn}d\Gamma _{j}  \notag \\
&&+(\beta ^{2}-1)\int\limits_{\Gamma _{j}}\left\vert h_{jn}\right\vert
^{2}d\Gamma _{j}-\int\limits_{\Gamma _{j}}\frac{\partial u_{n}}{\partial
\nu }h_{jn}d\Gamma _{j}  \notag \\
&&+\int\limits_{\Gamma _{j}}(\psi _{jn}^{\ast }+i\beta \varphi _{jn}^{\ast
})h_{jn}d\Gamma _{j}  \label{77}
\end{eqnarray}%
Here, we are also implicitly using $D(\mathbf{A})$-criterion (A.iv). For the first term on RHS: we note that upon combining the regularity for $D_s$ in (\ref{70}) with an integration by parts, we have that  
\begin{equation}
\frac{\partial }{\partial
\nu }D_s \in \mathcal{L}(H^{\frac{1}{2}}(\Gamma _{s}),H^{-\frac{1}{2}}(\Omega _{s})) \label{add1}
\end{equation}
This gives the estimate, via the decomposition (\ref{71}),
\begin{equation}
\left\Vert \frac{\partial w_{0n}}{\partial \nu }\right\Vert _{H^{-\frac{1}{2}%
}(\Gamma _{s})}\leq C\left[ \left\Vert \frac{\partial z_{n}}{\partial \nu }%
\right\Vert _{H^{-\frac{1}{2}}(\Gamma _{s})}+\left\Vert i\beta\frac{\partial
}{\partial \nu }D_{s}[u_{n}|_{\Gamma _{s}}+w_{0n}^{\ast }|_{\Gamma _{s}}]\right\Vert _{H^{-\frac{1}{2}%
}(\Gamma _{s})}\right] \leq C_{\beta },  \label{78}
\end{equation}%
after also using (\ref{61}), (\ref{67}), The Sobolev Trace Map, and (\ref{76}).
Applying this estimate to RHS of (\ref{77}), along with (\ref{68}), (\ref{69}%
), and (\ref{61}) we have%
\begin{equation}
\underset{n\rightarrow \infty }{\lim }h_{jn}=0\text{ \ \ in \ \ }%
H^{1}(\Gamma _{j}),\text{ \ \ }1\leq j\leq K.  \label{79}
\end{equation}%
\newline
\underline{\textbf{STEP 4:}} \\ \\ We note from the previous step that the limit
in (\ref{79}) when applied to the equation%
\begin{equation*}
\frac{\partial w_{0n}}{\partial \nu }|_{\Gamma _{j}}=-\Delta h_{jn}+(1-\beta
^{2})h_{jn}+\frac{\partial u_{n}}{\partial \nu }-(\psi _{jn}^{\ast }+i\beta
\varphi _{jn}^{\ast })\text{ \ \ in \ \ }\Gamma _{j},\text{ \ \ }1\leq j\leq
K,
\end{equation*}%
gives
\begin{equation}
\underset{n\rightarrow \infty }{\lim }\frac{\partial w_{0n}}{\partial \nu }%
|_{\Gamma _{j}}=0\text{ \ \ in \ \ }H^{-1}(\Gamma _{j}).  \label{80}
\end{equation}%
In obtaining this limit, along with (\ref{79}), we are also using (\ref{69})
and (\ref{61}). In turn, via an interpolation we have for $1\leq j\leq K,$%
\begin{eqnarray}
\left\Vert \frac{\partial z_{n}}{\partial \nu }\right\Vert _{H^{-\frac{1}{2}%
}(\Gamma _{j})} &\leq &C\left\Vert \frac{\partial z_{n}}{\partial \nu }%
\right\Vert _{H^{-1}(\Gamma _{j})}^{\frac{1}{2}}\left\Vert \frac{\partial
z_{n}}{\partial \nu }\right\Vert _{L^{2}(\Gamma _{j})}^{\frac{1}{2}}  \notag
\\
&=&C\left\Vert \frac{\partial w_{0n}}{\partial
\nu }+i\beta\frac{\partial
}{\partial \nu }D_{s}[u_{n}|_{\Gamma _{s}}+w_{0n}^{\ast }|_{\Gamma _{s}}]\right\Vert _{H^{-1}(\Gamma _{s})}^{\frac{1}{2}}\left\Vert 
\frac{\partial z_{n}}{\partial \nu }\right\Vert _{L^{2}(\Gamma _{j})}^{\frac{%
1}{2}}  \label{82}
\end{eqnarray}%
Applying (\ref{add1}), (\ref{61}), (\ref{80}) and (\ref{76}) to RHS of (%
\ref{82}), we have now (upon summing up over $j$),%
\begin{equation}
\underset{n\rightarrow \infty }{\lim }\frac{\partial z_{n}}{\partial \nu }=0%
\text{ \ \ in \ \ }H^{-\frac{1}{2}}(\Gamma _{s}) . \label{83}
\end{equation}%
\newline
\underline{\textbf{STEP 5}}: \ By (\ref{61}) we have that $\{z_{n}\}$ of (%
\ref{71}) converges weakly to, say, $z$ in $H_{0}^{1}(\Omega _{s}).$ With this
limit in mind, we multiply both sides of the wave equation in (\ref{bvp}) by
given $\eta \in H^{1}(\Omega _{s}).$ Integrating by parts we then have%
\begin{eqnarray*}
&&-\beta ^{2}(z_{n},\eta )_{\Omega _{s}}+(\nabla z_{n},\nabla \eta )_{\Omega
_{s}}+\left\langle \frac{\partial z_{n}}{\partial \nu },\eta \right\rangle
_{\Gamma _{s}} \\
&=&(w_{1n}^{\ast }+i\beta w_{0n}^{\ast }-i\beta D_{s}[u_{n}|_{\Gamma _{s}}+w_{0n}^{\ast }|_{\Gamma _{s}}],\eta )_{\Omega _{s}},\text{ \ \ \ }\forall \text{ }\eta \in
H^{1}(\Omega _{s}).
\end{eqnarray*}%
Taking the limit of both sides of this equation, while taking into account (%
\ref{61}), (\ref{67}), (\ref{70}), The Sobolev Trace Map, and (\ref{83}), we
obtain that $z\in $ $H_{0}^{1}(\Omega _{s})$ satisfies the variational
problem%
\begin{equation*}
-\beta ^{2}(z,\eta )_{\Omega _{s}}+(\nabla z,\nabla \eta )_{\Omega _{s}}=0,%
\text{ \ \ }\forall \text{ }\eta \in H^{1}(\Omega _{s})
\end{equation*}%
That is, $z$ satisfies the overdetermined eigenvalue problem%
\begin{equation*}
\left\{ 
\begin{array}{c}
-\Delta z=\beta ^{2}z\text{ \ \ in \ \ }\Omega _{s} \\ 
z|_{\Gamma _{s}}=\frac{\partial z}{\partial \nu }|_{\Gamma _{s}}=0%
\end{array}%
\right. 
\end{equation*}%
which gives that%
\begin{equation*}
z=0\text{ \ \ in \ \ }\Omega _{s}
\end{equation*}%
Combining this convergence with (\ref{71}), (\ref{67}), (\ref{61}) and (\ref{70}), we get%
\begin{equation}
\underset{n\rightarrow \infty }{\lim }w_{0n}=0\text{ \ \ in \ \ }%
H^{1}(\Omega _{s}).  \label{84}
\end{equation}%
\\
\textbf{Completion of the Proof of Lemma \ref{point-cont-spec}} \\

\noindent The resolvent relations in (\ref{63}), (\ref{64}) and the convergences (\ref{68}), (\ref{84})
give also%
\begin{equation}
\left\{ 
\begin{array}{c}
\underset{n\rightarrow \infty }{\lim }\xi _{jn}=0\text{ \ \ in \ \ }%
L^{2}(\Gamma _{j}),\text{ \ \ }1\leq j\leq K \\ 
\underset{n\rightarrow \infty }{\lim }w_{1n}=0\text{ \ \ in \ \ }%
H^{1}(\Omega _{s})%
\end{array}%
\right.   \label{85}
\end{equation}%
Collecting now, (\ref{67}), (\ref{79}), (\ref{84}) and (\ref{85}) we have%
\begin{equation*}
\underset{n\rightarrow \infty }{\lim }\Phi _{n}=0\text{ \ \ in \ \ }\mathbf{%
H,}
\end{equation*}%
which contradicts (\ref{61}) and finishes the proof of Lemma \ref{point-cont-spec}.\\

Lastly, we give the following Corollary regarding the residual spectrum $\sigma_r(\mathbf{A})$:

\begin{corollary}
\label{RS} The residual spectrum $\sigma_r(\mathbf{A})$ of $\mathbf{A}$ does not intersect the
imaginary axis.
\end{corollary}

\begin{proof}
Given the form of the adjoint operator $\mathbf{A}^{\ast }:\mathbf{%
H\rightarrow H}$ in Proposition \ref{adj}, then proceeding identically as in
the proof of Lemma \textbf{\ref{point-cont-spec}} we obtain 
\begin{equation*}
\sigma _{p}(\mathbf{A}^{\ast })\cap i%
\mathbb{R}
=\sigma _{c}(\mathbf{A}^{\ast })\cap i%
\mathbb{R}
=\emptyset 
\end{equation*}
which finishes the proof of Corollary \ref{RS}.
\end{proof}
\\

Now, having established the above results for the spectrum of $\mathbf{A}$, we are in a position to give the proof of Theorem \ref{SS}:\\

\textbf{Proof of Theorem \ref{SS}} \\

If we combine the above results Proposition \ref{invert}, Lemma \ref{point-cont-spec} and Corollary \ref{RS} and remember that $\left\{ e^{At}\right\} _{t\geq 0}$ is a contraction semigroup, the strong stability result follows immediately from the application of Theorem \ref{AB}.

\section{Acknowledgement}

The authors G. Avalos and Pelin G. Geredeli would like to thank the National Science Foundation, and acknowledge their partial funding from NSF Grant DMS-1616425 and NSF Grant DMS-1907823.

\noindent The author Boris Muha would like to thank the Croatian Science Foundation (Hrvatska zaklada za znanost), and acknowledge his partial funding from CSF Grant IP-2018-01-3706.

\end{document}